%% file: conAll.tex
\documentclass[paper=a4, fontsize=10pt]{amsart} 

\usepackage{mathalfa,amsmath,amsthm,amsfonts,amssymb}
\usepackage{framed,lambda}
\usepackage[english]{babel}
\usepackage{lipsum}
\usepackage[linktoc=all]{hyperref}
\usepackage[capitalise]{cleveref}
\usepackage{enumitem,imakeidx}
\usepackage{color,xcolor,graphicx,tikz,tikz-cd}
\usepackage{stmaryrd}
\usepackage[all]{xy}
\theoremstyle{plain}
\newtheorem{lemma}[subsection]{Lemma}
\newtheorem{theorem}[subsection]{Theorem}
\newtheorem*{theorem*}{Theorem}
\newtheorem{prop}[subsection]{Proposition}
\newtheorem{cor}[subsection]{Corollary}

\theoremstyle{definition}
\newtheorem{defn}[subsection]{Definition}

\theoremstyle{remark}

\newtheorem{remark}[subsection]{Remark}

\numberwithin{equation}{section} 
\numberwithin{figure}{section} 
\numberwithin{subsection}{section} 

\makeatletter
\let\c@equation\c@subsection
\makeatother

\makeatletter
\let\c@figure\c@subsection
\makeatother

\def\tensor{\mathop{\otimes}}
\def\mid{\,\middle\vert\,}

\newif\ifdraft
\draftfalse
\ifdraft
\usepackage{pgfpages}
\usepackage{marginnote}
\pgfpagesuselayout{resize to}[a3paper]
\pagecolor[rgb]{0,0,0}
\color[rgb]{0.0,0.9,0.4}
\fi
\draftfalse

\newif\ifsure
\surefalse

\input{definitions.tex}

\title{Conormal Varieties on the Cominuscule Grassmannian}

\author{V. Lakshmibai}
\author{Rahul Singh} 

\begin{document}

\input{abstract.tex}

\maketitle 


\input{intro.tex}
\input{weyl.tex}
\input{kmg.tex}
\input{good.tex}
\input{determinantal.tex}


\addcontentsline{toc}{section}{Bibliography}
\bibliographystyle{amsart}
\bibliography{biblio}

\noindent
 {\scshape V. Lakshmibai}\\
 {\scshape Department of Mathematics, Northeastern University,\\ Boston, MA 02115, USA}.\\
 {\itshape E-mail address}: \texttt{luckylakshmibai@gmail.com}

\noindent
 {\scshape Rahul Singh}\\
 {\scshape Department of Mathematics, Northeastern University,\\ Boston, MA 02115, USA}.\\
 {\itshape E-mail address}: \texttt{singh.rah@husky.neu.edu}
\end{document}

%% file: definitions.tex
\def\u{\ensuremath{\mathfrak u}}
\def\rw{\ensuremath R}
\def\lgl{\ensuremath{\mathfrak g}}
\def\gO{\ensuremath{L^+\gl}}
\def\gd{\ensuremath{{\mathcal P_d}}}
\def\con{\ensuremath{{N^*X_\J(w)}}}
\def\borel{\ensuremath B}
\def\Borel{\ensuremath{\mathcal B}}
\def\Para{\ensuremath{\mathcal P}}
\def\para{\ensuremath P} 
\def\gl{\ensuremath G}
\def\Ni{\ensuremath{\mathcal N}}
\def\gL{\ensuremath{L\gl}}
\def\roots#1.#2.{\ensuremath{\Delta^{#1}_{#2}}}
\def\proots{\ensuremath{\roots+.0.}}
\def\coroots{\ensuremath{\Lambda^\vee}}

\def\mod{\ensuremath{\mathrm{mod}\,}}
\def\dynk{\ensuremath{\mathcal D}}
\def\define{\ensuremath{\overset{\scriptscriptstyle\operatorname{def}}=}}
\def\ws{\ensuremath{w_\J}}
\def\J{\ensuremath{\mathcal J}}

\def\L{\ensuremath{\mathcal L}}

%% file: abstract.tex
\begin{abstract}
Let $G$ be a simply connected, almost simple group over an algebraically closed field $\mathbf k$, and $P$ a maximal parabolic subgroup corresponding to omitting a cominuscule root.
We construct a compactification $\phi:T^*G/P\rightarrow X(u)$, where $X(u)$ is a Schubert variety corresponding to the loop group $LG$.
Let $N^*X(w)\subset T^*G/P$ be the conormal variety of some Schubert variety $X(w)$ in $G/P$; hence we obtain that the closure of $\phi(N^*X(w))$ in $X(u)$ is a $B$-stable compactification of $N^*X(w)$. 
We further show that this compactification is a Schubert subvariety of $X(u)$ if and only if $X(w_0w)\subset G/P$ is smooth, where $w_0$ is the longest element in the Weyl group of $G$.
This result is applied to compute the conormal fibre at the zero matrix in any determinantal variety.
\end{abstract}

%% file: intro.tex
\section{Introduction}
Given a quiver $Q$, let $\overline Q=Q\sqcup Q^{op}$ be the \emph{double} of $Q$, the quiver that has the same vertex set as $Q$ and whose set of edges is a disjoint union of the sets of edges of $Q$ and of $Q^{op}$, the \emph{opposite} quiver.
Thus, for any edge $e\in Q$, there is also a reverse edge $e^*\in Q^{op}\subset\overline Q$ with the same endpoints as $e$, but in the opposite direction.
For $\bf d$ a dimension vector, the quiver variety $Rep_{\bf d}\overline Q$ is naturally identified as the cotangent bundle of $Rep_{\bf d}Q$ (cf. \cite{ginzburg2009lectures}).
\par 
\par 
Orbit closures in $Rep_{\bf d}A_2$ (see \Cref{TBL:comin}) are called \emph{determinantal varieties}. 
Lakshmibai and Seshadri \cite{lakshmibai1978geometry} have identified determinantal varieties as open subsets of certain \emph{Schubert varieties} in the type $A$ Grassmannian.
Further, Strickland \cite{strickland1982conormal} has identified the conormal varieties of determinantal varieties as certain \emph{nilpotent orbit closures} in $T^*Rep_{\bf d}A_2=Rep_{\bf d}\overline A_2$.
On the other hand, Lusztig \cite{lusztig1990canonical} has identified nilpotent orbit closures in $Rep_{\bf d}\overline A_2$ (and more generally in $Rep\,\widetilde A_n$) as open subsets of Schubert varieties in the affine Grassmannian. 
Inspired by these results, Lakshmibai \cite{lakshmibai2016cotangent} has suggested an exploration of the relationship between the conormal varieties of Schubert varieties and the corresponding affine type Schubert varieties. 
\par 
\par 
Let \gl\ be a simply connected, almost simple algebraic group over an algebraically closed field $\mathbf k$.
We identify \gl\ as a Kac-Moody group corresponding to some irreducible finite type Dynkin diagram $\dynk_0$.
The loop group $LG\define G\left(\mathbf k[t,t^{-1}]\right)$ is then a Kac-Moody group corresponding to the \emph{extended Dynkin diagram} \dynk, obtained by attaching to $\dynk_0$ the extra root $\alpha_0$ (see \Cref{TBL:comin}).
\par 
\par 
A simple root $\alpha_d$ is \emph{cominuscule} if and only if there exists an automorphism $\iota$ of \dynk\ such that $\iota(\alpha_0)=\alpha_d$ (see \Cref{TBL:comin}). 
Let \para\ be the parabolic subgroup in \gl\ corresponding to omitting a cominuscule root $\alpha_d$, and \Para\ the parabolic subgroup in \gL\ corresponding to omitting both $\alpha_d$ and $\alpha_0$.
For $\mathbf k=\mathbb C$, Lakshmibai \cite{lakshmibai2016cotangent}, and Lakshmibai, Ravikumar, and Slofstra \cite{lakshmibai2015cotangent} have constructed a dense embedding $\phi$ of $T^*G/P$ into a Schubert variety in $LG/\Para$.
We use the Kac-Moody functor of Tits \cite{tits1987uniqueness} to give a definition of $\phi$ which works in all characteristics (see \Cref{fibreId} and \Cref{form:phip}).
\par 
\par 
Let $w_0$ denote the longest element in the Weyl group of $G$, and \ws\ the longest element in the Weyl group of $P$.
Our main result (\Cref{mainResult}) is the following:
\begin{theorem}
\label{introMain}
Let $X(w)$ be a Schubert variety in $G/P$, and let \con\ be its conormal variety.
Then the closure of $\phi(\con)$ in $\gL/\Para$ is a Schubert variety if and only if the Schubert variety $X(w_0w\ws)$ in $G/P$ is smooth.
\end{theorem}
Billey and Mitchell \cite{billey2010smooth} have given a combinatorial criterion to identify smooth Schubert varieties in cominuscule Grassmannians (see \Cref{sb}).
Using this, we deduce that \Cref{introMain} applies to determinantal varieties, symmetric determinantal varieties, and skew-symmetric determinantal varieties, which Lakshmibai and Seshadri \cite{lakshmibai1978geometry} have identified as open subsets of certain Schubert sub-varieties of cominuscule Grassmannians of type $A$, $C$, and $D$ respectively.
\par 
\par 
Let $\overline\Sigma_r^{sk,n}$ denote the \emph{rank $r$ skew-symmetric determinantal variety}:
\begin{align*}
    \overline\Sigma_r^{sk,n}&=\left\{A\in Hom(k^n,k^n)\mid A=-A^T,\,rank(A)\leq r\right\} 
\end{align*}
Our second result (\Cref{fibreDet}) identifies the conormal fibre at $0$ of the skew-symmetric determinantal variety.
\begin{theorem}\label{thm1.2}
The conormal fibre of $\overline\Sigma_r^{sk,n}$ at $0$ is isomorphic to $\overline\Sigma_{\overline n-r}^{sk,n}$ where\begin{align*}
\overline n=\begin{cases} n&\text{if $n$ is even,}\\n-1&\text{if $n$ is odd.}\end{cases}
\end{align*}
\end{theorem}

Similar results are known for determinantal (see \cite{strickland1982conormal,gaffney2014pairs}) and symmetric determinantal varieties (see \cite{mich}).
Our proof of \Cref{fibreDet} can be adapted in a straightforward manner to recover these results.
\par 
\par 
The paper is arranged as follows.
In Section 2, we recall the basics of (finite and affine type) root systems and the corresponding almost simple groups.
We also describe how extending a finite type root system by attaching an extra root in a manner prescribed in \cite{kac1994infinite} corresponds to replacing the corresponding almost simple group with its loop group.
Finally, we recall some results on Weyl groups and Schubert varieties.
\par 
\par 
In Section 3, we show that the cotangent bundle of a cominuscule Grassmannian has a compactification $\phi$ by an affine Schubert variety.
Along the way, we construct (\Cref{defn:iota}) an involution $\iota$ of affine type Dynkin diagrams that exchanges a cominuscule root $\alpha_d$ with the \emph{extra} root $\alpha_0$.
The involution $\iota$ acts on the associated Weyl group by conjugation (\Cref{iotaConj}).
\par 
\par 
In Section 4, we study the conormal variety \con\ of a Schubert variety $X_\J(w)$ in a cominuscule Grassmannian $G/P$.
We leverage the main result of \cite{billey2010smooth} to develop characterizations of smooth Schubert varieties in $G/P$, see \Cref{sb}.
We then use this to prove that \con\ has a compactification as a Schubert variety via the embedding $\phi$ if and only if the Schubert variety $X_\J(w_0w\ws)$ is smooth, see \Cref{mainResult}.
This yields powerful results about the geometry of \con\ when $X_\J(w_0w\ws)$ is smooth, see \Cref{main1}.
Further, using Littelmann's work \cite{littelmann2003bases} on the standard monomial theory of affine Schubert varieties, we can write down the equations defining \con\ as a subvariety of $T^*G/P$, see \Cref{main2}.
In \Cref{conFibreId}, we give a description of the fibre of \con\ at identity as a union of Schubert varieties. 
\par 
\par 
In Section 5, we apply the results of Section 4 to skew-symmetric determinantal varieties.
The (usual, symmetric, skew-symmetric resp.) rank $r$ determinantal varieties can be identified as the opposite cells of certain Schubert varieties $X_\J(w_r)$ in certain cominuscule Grassmannian (of type $A$, $C$, $D$ resp.).
Working in type D, we first verify that $X_\J(w_0ww_r)$ is smooth (\Cref{wlj}); hence \Cref{conFibreId} applies.
We then make explicit computations in the Weyl group (\Cref{intersectw}) to prove that the fibre at the zero matrix of the skew-symmetric determinantal variety is the rank $\overline n-r$ skew-symmetric determinantal variety, see \Cref{thm1.2}.
\par 
\par 
\emph{Acknowledgments}: We thank Terence Gaffney for fruitful discussions that pointed us towards the results in \Cref{sec:det}.

%% file: weyl.tex
\ifdraft\marginpar{setup.tex}\fi

\section{Dynkin Diagrams and Weyl Groups} 
\label{setup}
In this section, we recall the basics of the theory of finite type and extended Dynkin diagrams, their root systems, and certain Kac-Moody groups associated to them.
Throughout, we assume that the base field $\mathbf k$ is algebraically closed.
The primary references for the combinatorial results in this section are \cite{BourbakiNicolas2008Lgal,kumar2012kac,tits1987uniqueness}.
For the geometric results, one may refer to \cite{faltings2003algebraic,kumar2012kac}.

\subsection{Finite Type Dynkin Diagrams}
Let $\dynk_0$ be an irreducible \emph{finite type} Dynkin diagram, and \roots.0. the \emph{abstract root system} associated to $\dynk_0$.
We shall denote by \roots+.0., \roots-.0., $\dynk_0$, $\theta_0$, $\mathbb Z\dynk_0$, and $W_0$, the positive roots, negative roots, simple roots, highest root, root lattice, and the Weyl group of \roots.0. respectively.

\subsection{Extended Dynkin Diagram}
We can attach a simple root $\alpha_0$ to $\dynk_0$ to get the \emph{extended Dynkin diagram} \dynk\ (see \cite{kac1994infinite},\Cref{TBL:comin}).
Let \roots.., \roots+.., \roots-.., \dynk, $\mathbb Z\dynk$, and $W$ denote the set of roots, positive roots, negative roots, simple roots, root lattice, and the Weyl group respectively of the abstract root system of \dynk.  

\subsection{Real and Imaginary Roots}
A root $\alpha\in\roots..$ is called a \emph{real root} if there exists $w\in W$ such that $w(\alpha)\in\dynk$; otherwise $\alpha$ is called an \emph{imaginary root}.
The root $\delta\define\alpha_0+\theta_0$ is called the \emph{basic imaginary root}.
The set of real (resp. positive) roots \roots.\mathrm{re}. (resp. \roots+..) has the following characterization in terms of $\delta$:\begin{align*}
    \roots.\mathrm{re}. &=\left\{\alpha+n\delta\mid\alpha\in\roots.0.,\,n\in\mathbb Z\right\}\\
    \roots+..           &=\left\{\alpha+n\delta\mid\alpha\in\roots.0.\sqcup\left\{0\right\},\,n>0\right\}\bigsqcup\roots+.0.
\end{align*}

\input{cotcomin_table-1.tex}

\subsection{Bruhat Order and Reduced Expressions}
The Weyl group $W$ is a \emph{Coxeter group} with \emph{simple reflections} $\left\{s_\alpha\mid\alpha\in\dynk\right\}$.
The \emph{Bruhat order} $\leq$ on $W$ is the partial order generated by the relations \begin{align}
\label{lem:minRule}
\hspace{60pt}   w s_\alpha>w&\iff w(\alpha)>0       &\forall\,\alpha\in\roots+..\\
\hspace{60pt}   s_\alpha w>w&\iff w^{-1}(\alpha)>0  &\forall\,\alpha\in\roots+..\nonumber
\end{align}
We say $w=s_1\ldots s_l$ is a \emph{reduced expression} for $w$ if each $s_i$ is a simple reflection, and any other expression $w=s'_1\ldots s'_k$ satisfies $k\geq l$. 
The \emph{length} $l(w)$ of an element $w\in W$ is the number of simple reflections in a reduced expression for $w$.
The length function satisfies the relation $v<w\implies l(v)<l(w)$.

\subsection{The Weyl Involution}
\label{weylInv}
The Weyl group $W_0$ is finite, and has a unique longest element $w_0$.
The element $w_0$ is an involution, i.e., $w_0^2=1$, and further satisfies $w_0(\roots+.0.)=\roots-.0.$.
It follows that $-w_0$ induces an involution of $\dynk_0$, called the \emph{Weyl involution} (see \cite[pg 158]{BourbakiNicolas2008Lgal}).

\subsection{Semi-Direct Product Decomposition}
\label{sdp}
Let $\coroots_0$ be the coroot lattice of \roots.0..
There exists (cf. \cite[\S 13.1.7]{kumar2012kac}) a group isomorphism $W\rightarrow W_0\ltimes\Lambda_0^\vee$ given by \begin{align*}
\hspace{70pt}   s_\alpha    &\mapsto (s_\alpha,0) &\text{for }\alpha\in\dynk_0\\
\hspace{70pt}   s_{\alpha_0}&\mapsto (s_\theta,-\theta^{\vee})&
\end{align*} 
where $\theta$ is the highest root in \roots.0..
For $q\in\Lambda_0^\vee$, we write $\tau_q\define(1,q)\in W_0\ltimes\Lambda^\vee_0$.
The action of $\tau_q$ on \roots.. is determined by the formula $\tau_q(\delta)=\delta$, and\begin{align}
\hspace{70pt}   \tau_q(\alpha)  &=\alpha-\alpha(q)\delta&\forall\,\alpha\in\roots.0.\label{form:tauq}
\end{align} 

\subsection{Support}
\label{support}
The \emph{support} of $w\in W$, denoted $Supp(w)$, is the smallest subset $\J\subset\dynk$ satisfying $w\in W_\J$. 
For $\alpha=\sum\limits_{\beta\in\dynk}a_\beta\beta$, we define the \emph{support} of $\alpha$ to be \begin{align*}
    Supp(\alpha)\define\left\{\beta\in\dynk\mid a_\beta\neq 0\right\}
\end{align*}
If $\alpha\in\roots+..$ and $w(\alpha)\in\roots-..$, then $Supp(\alpha)\subset Supp(w)$.
In particular, it follows from \Cref{eq:minRep} that 
$    W_\J\subset W^{\dynk\backslash\J}$ for any $\J\subset\dynk$.

\subsection{Minimal Representatives}
\label{minimalReps}
Let \J\ be some proper (necessarily finite type) sub-diagram of \dynk. 
We write \roots.\J., \roots+.\J., \roots-.\J., \J, and $W_\J$ for the set of roots, positive roots, negative roots, simple roots, and the Weyl subgroup respectively whose support is contained in \J.
Given an element $w\in W$, there exists a unique element $w^\J$, which is of minimal length in the coset $wW_\J$.
The element $w^\J$ is called the \emph{minimal representative} of $w$ with respect to \J.
The set of minimal representatives in $W$ with respect to \J\ is denoted $W^\J$.
It follows from \Cref{lem:minRule} that\begin{align}\label{eq:minRep}
    W^\J=\left\{w\in W\mid w(\alpha)>0,\,\forall\,\alpha\in\J\right\}
\end{align}

\subsection{The Group $G$}
Let \gl\ be the simply connected, almost simple algebraic group over $\mathbf k$ whose Dynkin diagram is $\dynk_0$.
We fix a \emph{torus} $T\subset G$, and a \emph{Borel subgroup} $B\subset G$ satisfying $T\subset B$.
We identify the root system of $(G,\borel,T)$ with the abstract root system \roots.0., and the Weyl group $W_0$ with $N/T$, where $N$ is the normalizer of $T$ in $G$.

\subsection{The Loop Group}
\label{def:kmg}
Let $\mathcal O\define\mathbf k[t]$, $\mathcal O^-\define\mathbf k[t^{-1}]$, and $\mathcal K\define\mathbf k[t,t^{-1}]$.
The \emph{loop group} $\gL\define G(\mathcal K)$ is a Kac-Moody group with Dynkin diagram \dynk, and is ind-representable by an affine scheme over $\mathbf k$.
Throughout, we shall identify the Weyl group $W$ of $\dynk$ with $N(\mathcal K)/T$.
\par 
\par 
Let \lgl\ be the Lie algebra of $G$.
We identify the Lie algebra $L\lgl$ of \gL\ with $\lgl\tensor\mathcal K$.
Let $U_\alpha$ denote the \emph{root subgroup} corresponding to a real root $\alpha\in\roots.\mathrm{re}.$ (see \cite{remy2002groupes,borel2012linear}).
We can identify \gl\ as the subgroup of \gL\ generated by $T$ and $\left\{U_\alpha\mid\alpha\in\roots.0.\right\}$. 
\par 
\par 
Let $\gO\define G(\mathcal O)$, $L^-\gl\define G(\mathcal O^-)$, and consider the \emph{evaluation maps}\begin{align*}
    \pi:\gO\rightarrow\gl,\quad t\mapsto 0   &&\pi_-: L^-\gl\rightarrow\gl,\quad t^{-1}\mapsto 0
\end{align*}
The subgroups $\Borel\define\pi^{-1}(\borel)$ and $\Borel^-\define\pi_-^{-1}(\borel^-)$ are called Borel subgroups of \gL. 
Suppose $\borel$, $\borel^-$ are \emph{opposite} in \gl, i.e., $\borel\cap\borel^-=T$. 
Then $\Borel$, $\Borel^-$ are opposite in \gL, i.e., $\Borel\cap\Borel^-=T$. 

\subsection{Nilpotent set of roots}
\label{nilpotentRoots}
(see \cite{tits1987uniqueness})
Let $\Psi$ be a finite set of real roots.
We say that $\Psi$ is \emph{prenilpotent} if there exist $w,w'\in W$ such that $w(\Psi)\subset\roots+..$ and $w'(\Psi)\subset\roots-..$. 
We say that $\Psi$ is \emph{closed} if\[
    \alpha,\beta\in\Psi,\ \alpha+\beta\in\roots..\implies \alpha+\beta\in\Psi.
\]
Finally, we say $\Psi$ is \emph{nilpotent} if it is pre-nilpotent and closed.
For $\alpha\in\Psi$, let $\lgl^\alpha$ be the associated root space (see for example \cite{kac1994infinite}).
If $\Psi$ is nilpotent, then so is the Lie sub-algebra\begin{align}\label{defgpsi}
    \lgl^\Psi\define\bigoplus\limits_{\alpha\in\Psi}{\lgl^\alpha} 
\end{align}

\subsection{Tits' Functor for Kac-Moody Groups}
\label{titsFunctor}
\vskip -2pt
For $\alpha$ a real root, let $U_\alpha$ be the group scheme over $\mathbf k$ isomorphic to $\mathbb G_a$ with Lie algebra $\lgl^\alpha$.
To every nilpotent set of roots $\Psi$, Tits\cite{tits1987uniqueness} associates a group scheme $U_\Psi$ that depends only on $\lgl^\Psi$, and is naturally a closed subgroup scheme of $\gL$ (cf. \cite{tits1987uniqueness}).
For any ordering of $\Psi$, the product morphism $$\prod\limits_{\alpha\in\Psi}U_\alpha\rightarrow U_\Psi$$ is a scheme isomorphism.
Hence, we get a scheme isomorphism $\eta:\lgl^\Psi\rightarrow U_\Psi$.

\subsection{Parabolic Subgroups}
For \J\ any subset of $\dynk$, the subgroup\begin{align*}
    \Para_\J\define\Borel W_\J\Borel=\left\{bnb'\mid b,b'\in\Borel,\,n\in N\left(\mathcal K\right),\,n\mod T\in W_\J\right\}
\end{align*}
is the \emph{parabolic subgroup} of \gL\ corresponding to $\J\subset\dynk$.
The parabolic subgroup $\Para_{\dynk_0}$ is precisely $\gO$.
For $\J\subset\dynk_0$, the subgroup\begin{align*}
    \para_\J\define\borel W_\J\borel=\left\{bnb'\mid b,b'\in\borel,\,n\in N,\,n\mod T\in W_\J\right\}
\end{align*}
is the parabolic subgroup of \gl\ corresponding to $\J\subset\dynk_0$.

{\def\para{\ensuremath{P_{\J}}}
\subsection{The Bruhat decomposition}
Consider some subset \J\ of $\dynk_0$.
The \emph{Bruhat decomposition} of $G$ is 
$$G=\bigsqcup\limits_{w\in W_0^\J}\borel w\para$$ where $W_0^\J\define W_0\cap W^\J$.
The \emph{partial flag variety} $G/P_\J$ has the decomposition
$$G/P_\J=\bigsqcup\limits_{w\in W_0^\J}\borel w\para\mod\para$$
Let $\leq$ denote the \emph{Bruhat order} on $W$.
For $w\in W_0^\J$, the \emph{Schubert variety} \[
    X_\J(w)\define\overline{\borel w\para}\mod\para=\bigsqcup\limits_{\substack{v\in W_0^\J\\ v\leq w}}\borel v\para\mod\para
\]
is a projective variety of dimension $l(w)$.
}

{\def\Para{\ensuremath{\mathcal P_\J}}
\subsection{Affine Schubert Varieties}
Fix a proper subset $\J\subsetneq\dynk$.
We have the \emph{Bruhat decomposition} \begin{align*}
    \gL         &=\bigsqcup\limits_{w\in W^\J}\Borel w\Para
\end{align*}
The quotient $\gL/\Para$ is an ind-scheme.
For $w\in W^\J$, the \emph{affine Schubert variety} $$X_\J(w)\define\overline{\Borel w\Para}\mod\Para$$ is a projective variety of dimension $l(w)$, and has the decomposition
$$X_\J(w)=\bigsqcup\limits_{\substack{v\leq w\\ v\in W^\J}}\Borel v\Para\mod\Para$$
Consider a proper subset $\mathcal L\subsetneq\dynk$, and let $w_{\mathcal L}$ be the longest element in $W_{\mathcal L}$.
The Schubert variety $X_\J(w_{\mathcal L}^\J)$ is $\mathcal P_{\mathcal L}$-homogeneous, hence smooth. 
Indeed, we have $\mathcal P_{\mathcal L}\cap\Para=\mathcal P_{\mathcal L\cap\J}$, and further, $X_\J(w_{\mathcal L}^\J)=\mathcal P_{\mathcal L}/\mathcal P_{\mathcal L\cap\J}$.

\subsection{The Opposite Cell}
\label{oppositeCell}
Let $\J,\mathcal L$ be as above.
Let $\boldsymbol e$ denote the image of the identity element in the quotient $\gL/\Para$.
The \emph{opposite cell} $X^-_\J(w)$, given by 
$$X^-_\J(w)\define\Borel^-\boldsymbol e\bigcap X_\J(w)$$ is an \emph{open affine} subvariety of $X_\J(w)$.
The opposite cell of the Schubert variety $X_\J(w_{\mathcal L}^\J)$ 
is isomorphic to the \emph{affine space} $\mathbb A^k$ for $k=Card(\roots-.\mathcal L.\backslash\roots.\J.)$.
Indeed, for any enumeration $\alpha_1,\ldots,\alpha_k$ of $\roots-.\mathcal L.\backslash\roots.\J.$, the map \begin{align}
\label{oppositeCell}
    U_{\alpha_1}\times\ldots\times U_{\alpha_k}&\longrightarrow \gL/\Para\\
    (u_1,\ldots,u_k)&\mapsto u_1\ldots u_k\,\mod\Para\nonumber
\end{align}
is an open immersion onto the opposite cell $X^-_\J(w_{\mathcal L}^\J)$.
We also mention here the \emph{$L^-\gl$ cell} $Y_\J(w)\subset X_\J(w)$,\begin{align}\label{bigLGcell}
    Y_\J(w)\define L^-\gl\boldsymbol e\bigcap X_\J(w)
\end{align}
which is also an \emph{open affine} subvariety in $X_\J(w)$. 
}

\subsection{The Demazure Product}
\label{lem:bruProd}
The \emph{Demazure product} $\star$ on $W$ is the unique associative product satisfying:
\begin{align}\label{formula:bruProd}
    s_\alpha\star w &=\begin{cases} 
                        w           &\text{if }s_\alpha w<w\\ 
                        s_\alpha w  &\text{if }s_\alpha w>w
                    \end{cases}
\end{align}
The double coset $\Borel w\Borel$ is called a \emph{Bruhat cell} in \gL.
Suppose $v=s_1\ldots s_k$ is a reduced presentation for $v\in W$. 
Then $v=s_1\star\ldots\star s_k$, and
\begin{align*}
    \Borel v\Borel  &=\Borel s_1\Borel\ldots\Borel s_k\Borel
\end{align*}
Consider $v\in W^\J$, $w\in W_\J$.
Then $v\star w=vw$.
More generally, we have\begin{align}\label{lengthVee}
    l(vw)=l(v)+l(w)\iff v\star w=vw
\end{align}

\begin{remark}
The reader may refer to \cite{knutson2004subword}, Remark 3.3 for further details, including a justification of the name ``Demazure product".
\end{remark}

%% file: cotcomin_table-1.tex
\begin{table}

  \begin{tikzpicture}[scale=.4]
    \draw (-1.5,0) node[anchor=east]  {$A_1$};
    \foreach \x in {4}
    \draw[thick,fill=black!70] (\x cm,0) circle (.3cm);
    \draw (4,.8) node {$\scriptscriptstyle{1}$};

    \draw (14.4,0) node[anchor=east]  {$\widetilde{A}_{1}$};
    \draw (25.2,0) node[anchor=west]  {\ };
    \foreach \x in {19,21}
    \draw[thick,fill=black!70] (\x cm,0) circle (.3cm);
    \foreach \y in {19.3}
    \draw[thick] (\y cm,0) -- +(1.4 cm,0);
    \draw (19,.8) node {$\scriptscriptstyle{1}$};
    \draw (21,.8) node {$\scriptscriptstyle{0}$};
    \draw (20,.4) node {$\scriptscriptstyle\infty$};
  \end{tikzpicture}

\vspace{3mm}

  \begin{tikzpicture}[scale=.4]
    \draw (-1.5,0) node[anchor=east]  {$A_{n}$};
    \foreach \x in {0,2,4,6,8}
    \draw[thick,fill=black!70] (\x cm,0) circle (.3cm);
    \draw[dotted, thick] (0.3 cm,0) -- +(1.4 cm,0);
    \foreach \y in {2.3,4.3}
    \draw[thick] (\y cm,0) -- +(1.4 cm,0);
    \draw[dotted, thick] (6.3 cm,0) -- +(1.4 cm,0);
    \draw (0,.8) node {$\scriptscriptstyle{1}$};
    \draw (2,.8) node {$\scriptscriptstyle{d-1}$};
    \draw (4,.8) node {$\scriptscriptstyle{d}$};
    \draw (6,.8) node {$\scriptscriptstyle{d+1}$};
    \draw (8,.8) node {$\scriptscriptstyle{n}$};

    \draw (14.4,0) node[anchor=east]  {$\widetilde{A}_{n}$};
    \draw (14,-1.3) node[anchor=east]  {$\scriptstyle{n\geq 2}$};
    \draw (25,0) node[anchor=west]  {\ };
    \foreach \x in {16,18,20,22,24}
    \draw[thick,fill=black!70] (\x cm,0) circle (.3cm);
    \draw[thick, fill=black!70] (20 cm,-2 cm) circle (.3cm);
    \draw[dotted,thick] (16.3 cm,0) -- +(1.4 cm,0);   
    \foreach \y in {18.3,20.3}
    \draw[thick] (\y cm,0) -- +(1.4 cm,0);
    \draw[dotted,thick] (22.3 cm,0) -- +(1.4 cm,0); 
    \draw[thick] (16.2 cm,-0.2 cm) -- +(3.5 cm,-1.65 cm); 
    \draw[thick] (20.3 cm,-1.9 cm) -- +(3.5 cm,1.65 cm); 
    \draw (16,.8) node {$\scriptscriptstyle{1}$};
    \draw (18,.8) node {$\scriptscriptstyle{d-1}$};
    \draw (20,.8) node {$\scriptscriptstyle{d}$};
    \draw (22,.8) node {$\scriptscriptstyle{d+1}$};
    \draw (24,.8) node {$\scriptscriptstyle{n}$};
    \draw (20,-1.2) node {$\scriptscriptstyle{0}$};
    
  \end{tikzpicture}
  \vspace{2mm}

  \begin{tikzpicture}[scale=.4]
    \draw (-2,0) node[anchor=east]  {$B_{n}$};
    \draw (-2,-1.3) node[anchor=east]  {$\scriptstyle{n\geq 2}$};
    \foreach \x in {1.5,3.5,5.5,7.5}
    \draw[thick,fill=white!70] (\x cm,0) circle (.3cm);
    \draw[thick,fill=black!70] (-0.5 cm,0) circle (.3cm);
    \draw[thick] (-0.2 cm,0) -- +(1.4 cm,0);
    \draw[dotted,thick] (1.8 cm,0) -- +(1.4 cm,0);
    \draw[thick] (3.8 cm,0) -- +(1.4 cm,0);
    \draw[thick] (5.8 cm, .1 cm) -- +(1.4 cm,0);
    \draw[thick] (5.8 cm, -.1 cm) -- +(1.4 cm,0);
    \draw[thick] (6.4 cm, .3 cm) -- +(.3 cm, -.3 cm);
    \draw[thick] (6.4 cm, -.3 cm) -- +(.3 cm, .3 cm);
    \draw (-0.5,.8) node {$\scriptscriptstyle{1}$};
    \draw (1.5,.8) node {$\scriptscriptstyle{2}$};
    \draw (3.5,.8) node {$\scriptscriptstyle{n-2}$};
    \draw (5.5,.8) node {$\scriptscriptstyle{n-1}$};
    \draw (7.5,.8) node {$\scriptscriptstyle{n}$};
    
    \draw (14,0) node[anchor=east]  {$\widetilde{B}_{n}$};
    \draw (25,0) node[anchor=west]  {\ };
    \draw (14,-1.5) node[anchor=east]  {$\scriptstyle{n\geq 3}$};
    \foreach \x in {18,20,22,24}
    \draw[thick,fill=white!70] (\x cm,0) circle (.3cm);
    \draw[xshift=15 cm,thick,fill=black!70] (30: 17 mm) circle (.3cm);
    \draw[xshift=15 cm,thick,fill=black!70] (-30: 17 mm) circle (.3cm);
    \draw[dotted,thick] (18.3 cm,0) -- +(1.4 cm,0);
    \foreach \y in {20.3}
    \draw[thick] (\y cm,0) -- +(1.4 cm,0);
    \draw[thick] (22.3 cm, .1 cm) -- +(1.4 cm,0);
    \draw[thick] (22.3 cm, -.1 cm) -- +(1.4 cm,0);
    \draw[thick] (22.9 cm, .3 cm) -- +(.3 cm, -.3 cm);
    \draw[thick] (22.9 cm, -.3 cm) -- +(.3 cm, .3 cm);
    \draw[xshift=17.5 cm,thick] (30: 3 mm) -- (136: 10.5 mm);
    \draw[xshift=17.5 cm,thick] (-30: 3 mm) -- (-136: 10.5 mm);
    \draw (16.5,1.5) node {$\scriptscriptstyle{1}$};
    \draw (16.5,-.2) node {$\scriptscriptstyle{0}$};
    \draw (18,.8) node {$\scriptscriptstyle{2}$};
    \draw (20,.8) node {$\scriptscriptstyle{n-2}$};
    \draw (22,.8) node {$\scriptscriptstyle{n-1}$};
    \draw (24,.8) node {$\scriptscriptstyle{n}$};
  \end{tikzpicture}
  
    \vspace{3mm}

  \begin{tikzpicture}[scale=.4]
    \draw (-1,0) node[anchor=east]  {$C_{n}$};
    \draw (-1,-1.3) node[anchor=east]  {$\scriptstyle{n\geq 2}$};
    \foreach \x in {0.5,2.5,4.5,6.5}
    \draw[thick,fill=white!70] (\x cm,0) circle (.3cm);
    \draw[thick,fill=black!70] (8.5 cm,0) circle (.3cm);
    \draw[thick] (0.8 cm,0) -- +(1.4 cm,0);
    \draw[dotted,thick] (2.8 cm,0) -- +(1.4 cm,0);
    \draw[thick] (4.8 cm,0) -- +(1.4 cm,0);
    \draw[thick] (6.8 cm, .1 cm) -- +(1.4 cm,0);
    \draw[thick] (6.8 cm, -.1 cm) -- +(1.4 cm,0);
    \draw[thick] (7.4 cm, 0 cm) -- +(.3 cm, .3 cm);
    \draw[thick] (7.4 cm, 0 cm) -- +(.3 cm, -.3 cm);
    \draw (0.5,.8) node {$\scriptscriptstyle{1}$};
    \draw (2.5,.8) node {$\scriptscriptstyle{2}$};
    \draw (4.5,.8) node {$\scriptscriptstyle{n-2}$};
    \draw (6.5,.8) node {$\scriptscriptstyle{n-1}$};
    \draw (8.5,.8) node {$\scriptscriptstyle{n}$};
    
    \draw (15,0) node[anchor=east]  {$\widetilde{C}_{n}$};
    \draw (15,-1.3) node[anchor=east]  {$\scriptstyle{n\geq 2}$};
    \foreach \x in {16,18,20,22,24,26}
    \draw[thick,fill=white!70] (\x cm,0) circle (.3cm);
    \foreach \x in {16,26}
    \draw[thick,fill=black!70] (\x cm,0) circle (.3cm);
    \draw[dotted,thick] (20.3 cm,0) -- +(1.4 cm,0);
    \foreach \y in {18.3,22.3}
    \draw[thick] (\y cm,0) -- +(1.4 cm,0);
    \draw[thick] (24.3 cm, .1 cm) -- +(1.4 cm,0);
    \draw[thick] (24.3 cm, -.1 cm) -- +(1.4 cm,0);
    \draw[thick] (24.9 cm, 0 cm) -- +(.3 cm, .3 cm);
    \draw[thick] (24.9 cm, 0 cm) -- +(.3 cm, -.3 cm);
    \draw[thick] (16.3 cm, .1 cm) -- +(1.4 cm,0);
    \draw[thick] (16.3 cm, -.1 cm) -- +(1.4 cm,0);
    \draw[thick] (17.2 cm, 0 cm) -- +(-.3 cm, .3 cm);
    \draw[thick] (17.2 cm, 0 cm) -- +(-.3 cm, -.3 cm);
    \draw (16,.8) node {$\scriptscriptstyle{0}$};
    \draw (18,.8) node {$\scriptscriptstyle{1}$};
    \draw (20,.8) node {$\scriptscriptstyle{2}$};
    \draw (22,.8) node {$\scriptscriptstyle{n-2}$};
    \draw (24,.8) node {$\scriptscriptstyle{n-1}$};
    \draw (26,.8) node {$\scriptscriptstyle{n}$};
  \end{tikzpicture}

  \vspace{3mm}

  \begin{tikzpicture}[scale=.4]
    \draw (-1,0) node[anchor=east]  {$D_{n}$};
    \draw (-1,-1.3) node[anchor=east]  {$\scriptstyle{n\geq 4}$};
    \foreach \x in {2,4,6,8}
    \draw[thick,fill=white!70] (\x cm,0) circle (.3cm);
    \draw[thick,fill=black!70] (0 cm,0) circle (.3cm);
    \draw[xshift=8 cm,thick,fill=black!70] (30: 17 mm) circle (.3cm);
    \draw[xshift=8 cm,thick,fill=black!70] (-30: 17 mm) circle (.3cm);
    \draw[dotted,thick] (4.3 cm,0) -- +(1.4 cm,0);
    \foreach \y in {0.3, 2.3,6.3}
    \draw[thick] (\y cm,0) -- +(1.4 cm,0);
    \draw[xshift=8 cm,thick] (30: 3 mm) -- (30: 14 mm);
    \draw[xshift=8 cm,thick] (-30: 3 mm) -- (-30: 14 mm);
    \draw (0,.8) node {$\scriptscriptstyle{1}$};
    \draw (2,.8) node {$\scriptscriptstyle{2}$};
    \draw (4,.8) node {$\scriptscriptstyle{3}$};
    \draw (6,.8) node {$\scriptscriptstyle{n-3}$};
    \draw (8,.8) node {$\scriptscriptstyle{n-2}$};
    \draw (9.45,1.6) node {$\scriptscriptstyle{n-1}$};
    \draw (9.45,-.2) node {$\scriptscriptstyle{n}$};
    
    \draw (15,0) node[anchor=east]  {$\widetilde{D}_{n}$};
    \draw (15,-1.3) node[anchor=east]  {$\scriptstyle{n\geq 4}$};
    \foreach \x in {18,20,22,24}
    \draw[thick,fill=white!70] (\x cm,0) circle (.3cm);
    \draw[xshift=24 cm,thick,fill=black!70] (30: 17 mm) circle (.3cm);
    \draw[xshift=24 cm,thick,fill=black!70] (-30: 17 mm) circle (.3cm);
    \draw[xshift=15 cm,thick,fill=black!70] (30: 17 mm) circle (.3cm);
    \draw[xshift=15 cm,thick,fill=black!70] (-30: 17 mm) circle (.3cm);
    \draw[dotted,thick] (20.3 cm,0) -- +(1.4 cm,0);
    \foreach \y in {18.3, 22.3}
    \draw[thick] (\y cm,0) -- +(1.4 cm,0);
    \draw[xshift=24 cm,thick] (30: 3 mm) -- (30: 14 mm);
    \draw[xshift=24 cm,thick] (-30: 3 mm) -- (-30: 14 mm);
    \draw[xshift=17.5 cm,thick] (30: 3 mm) -- (136: 10.5 mm);
    \draw[xshift=17.5 cm,thick] (-30: 3 mm) -- (-136: 10.5 mm);
    \draw (16.5,1.5) node {$\scriptscriptstyle{1}$};
    \draw (16.5,-.2) node {$\scriptscriptstyle{0}$};
    \draw (18,.8) node {$\scriptscriptstyle{2}$};
    \draw (20,.8) node {$\scriptscriptstyle{3}$};
    \draw (22,.8) node {$\scriptscriptstyle{n-3}$};
    \draw (24,.8) node {$\scriptscriptstyle{n-2}$};
    \draw (25.45,1.6) node {$\scriptscriptstyle{n-1}$};
    \draw (25.45,-.2) node {$\scriptscriptstyle{n}$};
  \end{tikzpicture}

    \vspace{1mm}

  \begin{tikzpicture}[scale=.4]
    \draw (-2.4,0) node[anchor=east]  {$E_{6}$};
    \foreach \x in {1,3,5}
    \draw[thick,fill=white!70] (\x cm,0) circle (.3cm);
    \draw[thick,fill=white!70] (3 cm, 2 cm) circle (.3cm);
    \foreach \x in {-1,7}
    \draw[thick,fill=black!70] (\x cm,0) circle (.3cm);
    \foreach \y in {-0.7, 1.3, 3.3, 5.3}
    \draw[thick] (\y cm,0) -- +(1.4 cm,0);
    \draw[thick] (3 cm,.3 cm) -- +(0,1.4 cm);
    \draw (-1,.8) node {$\scriptscriptstyle{1}$};
    \draw (1,.8) node {$\scriptscriptstyle{3}$};
    \draw (3.5,.8) node {$\scriptscriptstyle{4}$};
    \draw (5,.8) node {$\scriptscriptstyle{5}$};
    \draw (7,.8) node {$\scriptscriptstyle{6}$};
    \draw (3.5,2.8) node {$\scriptscriptstyle{2}$};
    
    \draw (13.6,0) node[anchor=east] {$\widetilde{E}_{6}$};
    \foreach \x in {18,20,22}
    \draw[thick,fill=white!70] (\x cm,0) circle (.3cm);
    \draw[thick,fill=white!70] (20 cm, 2 cm) circle (.3cm);
    \draw[thick,fill=black!70] (20 cm, 4 cm) circle (.3cm);
    \draw[thick,fill=black!70] (16 cm,0) circle (.3cm);
    \draw[thick,fill=black!70] (24 cm,0) circle (.3cm);
    \foreach \y in {16.3, 18.3, 20.3, 22.3}
    \draw[thick] (\y cm,0) -- +(1.4 cm,0);
    \draw[thick] (20 cm,.3 cm) -- +(0,1.4 cm);
    \draw[thick] (20 cm,2.3 cm) -- +(0,1.4 cm);
    \draw (16,.8) node {$\scriptscriptstyle{1}$};
    \draw (18,.8) node {$\scriptscriptstyle{3}$};
    \draw (20.5,.8) node {$\scriptscriptstyle{4}$};
    \draw (22,.8) node {$\scriptscriptstyle{5}$};
    \draw (24,.8) node {$\scriptscriptstyle{6}$};
    \draw (20.5,2.8) node {$\scriptscriptstyle{2}$};
    \draw (20.5,4.8) node {$\scriptscriptstyle{0}$};
  \end{tikzpicture}

\vspace{3mm}

  \begin{tikzpicture}[scale=.4]
    \draw (-1,0) node[anchor=east]  {$E_{7}$};
    \foreach \x in {0,2,4,6,8}
    \draw[thick,fill=white!70] (\x cm,0) circle (.3cm);
    \draw[thick,fill=white!70] (4 cm, 2 cm) circle (.3cm);
    \draw[thick,fill=black!70] (10 cm,0) circle (.3cm);
    \foreach \y in {0.3, 2.3, 4.3, 6.3,8.3}
    \draw[thick] (\y cm,0) -- +(1.4 cm,0);
    \draw[thick] (4 cm,.3 cm) -- +(0,1.4 cm);
    \draw (0,.8) node {$\scriptscriptstyle{1}$};
    \draw (2,.8) node {$\scriptscriptstyle{3}$};
    \draw (4.5,.8) node {$\scriptscriptstyle{4}$};
    \draw (6,.8) node {$\scriptscriptstyle{5}$};
    \draw (8,.8) node {$\scriptscriptstyle{6}$};
    \draw (10,.8) node {$\scriptscriptstyle{7}$};
    \draw (4.5,2.8) node {$\scriptscriptstyle{2}$};
    
    \draw (14,0) node[anchor=east]  {$\widetilde{E}_{7}$};
    \foreach \x in {17,19,21,23,25}
    \draw[thick,fill=white!70] (\x cm,0) circle (.3cm);
    \draw[thick,fill=white!70] (21 cm, 2 cm) circle (.3cm);
    \draw[thick,fill=black!70] (15 cm,0) circle (.3cm);
    \draw[thick,fill=black!70] (27 cm,0) circle (.3cm);
    \foreach \y in {15.3, 17.3, 19.3, 21.3,23.3,25.3}
    \draw[thick] (\y cm,0) -- +(1.4 cm,0);
    \draw[thick] (21 cm,.3 cm) -- +(0,1.4 cm);
    \draw (15,.8) node {$\scriptscriptstyle{0}$};
    \draw (17,.8) node {$\scriptscriptstyle{1}$};
    \draw (19,.8) node {$\scriptscriptstyle{3}$};
    \draw (21.5,.8) node {$\scriptscriptstyle{4}$};
    \draw (23,.8) node {$\scriptscriptstyle{5}$};
    \draw (25,.8) node {$\scriptscriptstyle{6}$};
    \draw (27,.8) node {$\scriptscriptstyle{7}$};
    \draw (21.5,2.8) node {$\scriptscriptstyle{2}$};
  \end{tikzpicture}
  
 \vspace{3mm} 
    \caption{(see \cite{BourbakiNicolas2008Lgal}) Finite type Dynkin diagrams with cominuscule simple roots marked in black (left column), and the corresponding extended Dynkin diagrams with all the cominuscule roots and the additional affine root marked in black (right column).}
    \label{TBL:comin}
\end{table}

%% file: kmg.tex
\section{The Cominuscule Grassmannian}
Let $\dynk_0$ be a finite type Dynkin diagram, and \dynk\ its associated extended Dynkin diagram.
In this section, we first recall the notion of cominuscule roots and cominuscule Grassmannians associated to $\dynk_0$.
We then fix a choice of cominuscule root $\alpha_d$, and develop the rest of this section (and the next) for that fixed choice of $\alpha_d$.
We introduce in \Cref{defn:iota} a canonical Dynkin diagram involution $\iota$ depending only on $\alpha_d$.
\par 
\par 
Let $G$ be the almost simple, simply connected algebraic group with Dynkin diagram $\dynk_0$, and $P$ the maximal parabolic subgroup corresponding to `omitting' the simple root $\alpha_d$.
In characteristic $0$, Lakshmibai, Ravikumar, and Slofstra \cite{lakshmibai2015cotangent} have constructed an isomorphism $\phi$ of $T^*G/P$ with the opposite cell of an affine Schubert variety in $LG/\Para$ (see \Cref{defnPara} for the definition of \Para). 
We give an alternate description of this Schubert variety in \Cref{defn:q}.
In \Cref{fibreId}, we give a characteristic free definition of $\phi$.

\begin{defn}
A simple root $\alpha_d\in\dynk_0$ is called \emph{cominuscule} if the coefficient of $\alpha_d$ in $\delta$ is $1$. 
\end{defn}

Observe from \Cref{TBL:comin} that a simple root $\alpha_d$ is cominuscule if and only if there exists an automorphism $\iota$ of \dynk\ such that $\iota(\alpha_0)=\alpha_d$.
For the remainder of this section (and the next), let $\alpha_d$ be some \emph{fixed cominuscule root} in $\dynk_0$, and set\begin{align*}
    \dynk_d\define\dynk\backslash\left\{\alpha_d\right\}&&\J\define\dynk_0\cap\dynk_d&&\theta_d\define\delta-\alpha_d
\end{align*}
We write \roots.d., \roots+.d., \roots-.d., $\mathbb Z\dynk_d$, and $W_d$ for the set of roots, positive roots, negative roots, root lattice, and Weyl group respectively of the root system associated to $\dynk_d$.
Observe that $\theta_d$ is the highest root of the finite type root system \roots.d..

\subsection{The Cominuscule Grassmannian}
\label{defnPara}
Let \gd\ be the parabolic subgroup corresponding to the set of simple roots $\dynk_d$.
We will write \para, \Para\ for the parabolic subgroups $\para_\J\subset\gl$ and $\Para_\J\subset\gL$ respectively.
Observe that $\Para=\gO\cap\gd$ and $\para=\gl\cap\gd$.
The variety $\gl/\para$ is called a \emph{cominuscule Grassmannian} of type $\dynk_0$. 

\subsection{The Cotangent Space}
Let $\lgl$, $\mathfrak p$, $\mathfrak h$ denote the Lie algebras of \gl, \para, $T$ respectively. 
We have the \emph{root space decompositions}\begin{align*}
    \lgl            &=\mathfrak h\oplus\bigoplus\limits_{\alpha\in\roots.0.}\lgl^\alpha\\
    \mathfrak p     &=\mathfrak h\oplus\bigoplus\limits_{\alpha\in\roots+.0.\sqcup\roots-.J.}\lgl^\alpha
\end{align*}
Let us identify \lgl\ with its dual using the \emph{Killing form} (cf. \cite{MR1808366}).
In particular, the dual of a root space $\lgl^\alpha$ is identified with the root space $\lgl^{-\alpha}$.
Now the tangent space at identity of $G/P$ being $\lgl/\mathfrak p$, we can identify its dual with 
$$\bigoplus\limits_{\alpha\in\roots-.0.\backslash\roots-.\J.}\lgl^{-\alpha}=\bigoplus\limits_{\alpha\in\roots+.0.\backslash\roots+.\J.}\lgl^{\alpha}=\u_P$$
where $\u_P$ is the Lie algebra of the unipotent radical $U_P$ of \para.

\subsection{The Cotangent Bundle}
The cotangent bundle $T^*\gl/\para$ is a vector bundle over $\gl/\para$, the fibre at any point $x\in\gl$ being the cotangent space to $\gl/\para$ at $x$; the dimension of $T^*\gl/\para$ equals $2\dim\gl/\para$.
Also, $T^*\gl/\para$ is the fibre bundle over $\gl/\para$ associated to the principal \para-bundle $\gl\rightarrow\gl/\para$ for the adjoint action of \para\ on $\u_\para$.
Thus \begin{align*}
    T^*\gl/\para=\gl\times^\para\u_\para&=\gl\times\u_\para/\sim
\end{align*}
where $\sim$ is the equivalence relation given by $(gp,u)\sim(g,pup^{-1})$ for $g\in\gl,\,u\in\u_\para,\ p\in\para$.  

\begin{lemma}\label{wsontheta}
Let $\ws$ be the longest element of $W_\J$.
We have $\ws(\alpha_d)=\theta_0$ and $\ws(\alpha_0)=\theta_d$.
\end{lemma}
\begin{proof}
To show $\ws(\alpha_d)=\theta_0$, it is enough to show that $\ws(\alpha_d)$ is maximal in $\roots+.0.\backslash\roots+.\J.$.
Observe first that $\ws(\roots.0.)=\roots.0.$, $\ws(\roots.\J.)=\roots.\J.$, and 
$$\left\{\alpha\in\roots+.0.\mid \ws (\alpha)<0\right\}=\roots+.\J.$$
Consequently, $\ws(\roots+.0.\backslash\roots+.\J.)\subset\roots+.0.\backslash\roots+.\J.$.
Consider $\alpha\in\roots+.0.\backslash\roots+.\J.$, and let $\gamma=\alpha-\alpha_d$.
Observe that $\alpha_d\leq\alpha$; hence $\gamma\geq 0$.
Further, since $\alpha_d$ is cominuscule, we have $2\alpha_d\not\leq\alpha$.
It follows that $\alpha_d\not\in Supp(\gamma)$, and so $Supp(\gamma)\subset \J$. 
Hence\begin{align*}
    \ws(\gamma)\leq 0\implies\ws(\alpha)&=\ws(\alpha_d)+\ws(\gamma)\\
                     \implies\ws(\alpha)&\leq\ws(\alpha_d)
\end{align*}
We see that $\ws (\alpha_d)$ is maximal in $\roots+.0.\backslash\roots+.\J.$, hence $\ws(\alpha_d)=\theta_0$.
The formula $\ws(\alpha_0)=\theta_d$ follows similarly, by showing that $\ws(\alpha_0)$ is maximal in $\roots+.d.\backslash\roots+.\J.$.
\end{proof}

\subsection{Bilinear Form}
Let $V$ denote the real vector space 
with basis \dynk. 
There exists a $W$-invariant symmetric bilinear form $\left(\ \mid\ \right)$ on $V$ (cf. \cite[\S3.7]{kac1994infinite}) such that\begin{align}\label{biFormRef}
    s_\alpha(\beta)=\beta-2\frac{\left(\alpha\mid\beta\right)}{\left(\alpha\mid\alpha\right)}\alpha.
\end{align}

\begin{defn}
[The Involution $\iota$]
\label{defn:iota}
Let $\iota$ be the linear involution of $V$ given by\begin{align*}
    \iota(\alpha)   &=\begin{cases}\alpha_d&\text{for }\alpha=\alpha_0\\ 
                        \alpha_0    &\text{for }\alpha=\alpha_d\\ 
                        -\ws(\alpha)&\text{for }\alpha\in\J
                    \end{cases}
\end{align*}
\end{defn}

\begin{lemma}
\label{ruleInv}
The form $\left(\ \mid\ \right)$ is invariant under $\iota\in GL(V)$.
\end{lemma}
\begin{proof}
Recall that 
$\left(\delta\mid\ \right)=0$ (cf. \cite[\S5.2]{kac1994infinite}).
Given $\alpha,\beta\in\J$, we have\begin{align*}
    \left(\iota(\alpha)\mid\iota(\beta)\right)  &=\left(-\ws(\alpha)\mid-\ws(\beta)\right)=\left(\alpha\mid\beta\right)\\
    \left(\iota(\alpha_0)\mid\iota(\beta)\right)    &=\left(\alpha_d\mid-\ws(\beta)\right) =\left(\ws(\theta_0)\mid-\ws(\beta)\right) \\
    &=\left(-\theta_0\mid\beta\right)=\left(\alpha_0-\delta\mid\beta\right)=\left(\alpha_0\mid\beta\right)\\
    \left(\iota(\alpha_d)\mid\iota(\beta)\right)    &=\left(\alpha_0\mid-\ws(\beta)\right) =\left(\ws(\theta_d)\mid-\ws(\beta)\right) \\
    &=\left(-\theta_d\mid\beta\right)=\left(\alpha_d-\delta\mid\beta\right)=\left(\alpha_d\mid\beta\right)\\
   \left(\iota(\alpha_0)\mid\iota(\alpha_d)\right)  &=\left(\alpha_d\mid\alpha_0\right)=\left(\alpha_0\mid\alpha_d\right)
\end{align*}
\vskip -2ex
\end{proof}

\begin{prop}
\label{FormInv}
The map $\iota$ induces an involution of the Dynkin diagram $\dynk$.
\end{prop}
\begin{proof}
It is clear from the definition that $\iota$ is an involution.
Further, since $-\ws$ induces an involution of $\J$ that preserves its Dynkin diagram structure (cf. \cite[pg 158]{BourbakiNicolas2008Lgal}), it follows that $\iota$ preserves the set of simple roots $\dynk$.
Now, it follows from \begin{align*}
    \alpha^\vee_i(\alpha_j)=\dfrac{\left(\alpha_i\mid\alpha_j\right)}{\left(\alpha_i\mid\alpha_i\right)}
\end{align*}
that the \emph{Cartan matrix} $\left(\alpha^\vee_i(\alpha_j)\right)_{ij}$ is preserved under $\iota$, and so $\iota$ preserves the Dynkin diagram structure on \dynk.
\end{proof}

\begin{cor}
\label{iotadelta}
We have the equality $\iota(\delta)=\delta$.
\end{cor}
\begin{proof}
We see from $\left(\delta\mid\ \right)=0$ and \Cref{ruleInv} that $\left(\iota(\delta)\mid\ \right)=0$.
Hence $\iota(\delta)=k\delta$ for some $k\in\mathbb Z$ (cf. \cite[\S5.6]{kac1994infinite}).
Further it follows from \Cref{FormInv} that $k>0$ and $k^2=1$.
We deduce that $k=1$, i.e., $\iota(\delta)=\delta$.
\end{proof}

\subsection{Action on $W$}
\label{actionW}
We also define an involution $\iota$ of $W$ given by \begin{align*}
\hspace{90pt}   {}^\iota s_\alpha\define s_{\iota(\alpha)}&&\text{for }\alpha\in\dynk
\end{align*}
It is clear that $\iota$ preserves the length and the Bruhat order on $W$, and\begin{align*}
    {}^\iota W_{\J} &=W_{\J}&{}^\iota W^\J&=W^\J&{}^\iota \ws&=\ws\\
    {}^\iota W_0    &=W_d   &{}^\iota W^0 &=W^d &{}^\iota w_d&=w_0
\end{align*}
Using \Cref{ruleInv} and \Cref{biFormRef}%
, we see that\begin{align*}
    \iota(s_\alpha(\iota(\beta)))=\iota\left(\iota(\beta)-2\frac{\left(\alpha\mid\iota(\beta)\right)}{\left(\alpha\mid\alpha\right)}\alpha\right) 
            =\beta-2\frac{\left(\iota(\alpha)\mid\beta\right)}{\left(\iota(\alpha)\mid\iota(\alpha)\right)}\iota(\alpha)=s_{\iota(\alpha)}(\beta)
\end{align*}
It follows that the action of $\iota$ on $w$ is the same as conjugation by $\iota$, where both $w$ and $\iota$ are viewed as elements of $GL(V)$, i.e.,\begin{align}\label{iotaConj}
    {}^\iota w=\iota w\iota
\end{align}
Note also that since Schubert varieties depend only on the underlying Dynkin diagrams, there exists an isomorphism $X_\J({}^\iota w)\cong X_\J(w)$ for any $w\in W$.

\subsection{The element $\tau_q$}
Let $w_0,\,w_d$ be the maximal elements in $W_0,\,W_d$ respectively, and let $\varpi_d^\vee$ be the fundamental co-weight dual to $\alpha_d$. 
Set\begin{align}
\label{defn:q}
    q\define w_0(\varpi_d^\vee)-\varpi_d^\vee\in\coroots_0,
\end{align}
and let $\tau_q\in W$ be the element corresponding to $q\in\coroots_0$ (see \Cref{sdp}).

\begin{prop}
\label{result:q}
We have the equality $\tau_q=w_0\ws w_d\ws=w_0^\J w_d^\J$.
\end{prop}
\begin{proof}
The set of roots \roots.. is contained in the $\mathbb Z$-span of the set $\dynk_0\cup\left\{\delta\right\}$.
Since the action of $W$ on \roots.. is faithful, it is enough to verify \begin{align}\label{work0}
\hspace{60pt}   &w_0\ws w_d\ws(\alpha)=\tau_q(\alpha)   &\forall\,\alpha\in\dynk_0\cup\left\{\delta\right\}
\end{align}
Further, since $\delta$ is fixed under the action of $W$ (cf. \Cref{sdp}), it is sufficient to verify \Cref{work0} for $\alpha\in\dynk_0$.
\par 
\par 
Recall from \Cref{weylInv} $-w_0$ induces an involution on $\dynk_0$.
Set $\beta=-w_0(\alpha_d)$, so that $-w_0(\varpi_d^\vee)=\varpi_\beta^\vee$, the fundamental co-weight dual to $\beta$.
It follows from \Cref{form:tauq,defn:q} that \begin{align}\label{work2}
    \tau_q(\alpha)  &=\alpha+\alpha(\varpi_d^\vee)\delta-\alpha(w_0(\varpi_d^\vee))\delta\\
                    &=\alpha+\alpha(\varpi_d^\vee+\varpi_\beta^\vee)\delta\nonumber
\end{align}
Hence we can rewrite \Cref{work0} as \begin{align}\label{work1}
\hspace{20pt}   w_0\ws w_d\ws(\alpha)   &=\alpha+\alpha(\varpi_d^\vee+\varpi_\beta^\vee)\delta  &\forall\,\alpha\in\dynk_0
\end{align}
Next, it follows from $\beta=-w_0(\alpha_d)$ that \begin{align}
        &\iota(\beta)=-\iota(w_0(\alpha_d))=-{}^\iota w_0(\iota(\alpha_d))&\text{using \Cref{iotaConj}}\nonumber\\
\implies&\iota(\beta)=-w_d(\alpha_0)&\text{using \Cref{actionW}}\label{work4}
\end{align}
Further, we have \begin{align}\hspace{40pt}
\label{work5}   w_d\ws(\alpha_d)&=w_d(\theta_0)=w_d(\delta-\alpha_0)        &\text{using \Cref{wsontheta}}\\
\nonumber                       &=\delta-w_d(\alpha_0)=\delta+\iota(\beta)  &\text{using \Cref{work4}}\\ 
\label{work6}   w_0\ws(\alpha_0)&={}^\iota(w_d\ws)(\iota(\alpha_d))         &\text{using \Cref{actionW}}\\
\nonumber                       &=\iota(\delta+\iota(\beta))=\delta+\beta   &\text{using \Cref{iotadelta}}
\end{align}   
We are now ready to prove that \Cref{work1} holds for $\alpha\in\left\{\alpha_d,\beta\right\}$.
\begin{itemize}
    \item[{\bf Case 1}] Suppose $\beta=\alpha_d$.
        Then $\iota(\beta)=\alpha_0$, $\varpi_d^\vee=\varpi_\beta^\vee$, and $q=\tau_{2\varpi_d^\vee}$.
        We have:\begin{align*}
            w_0\ws w_d\ws(\alpha_d) &=w_0\ws(\delta+\iota(\beta))   &\text{using \Cref{work5}}\\
                &=w_0\ws(\delta+\alpha_0)               &\text{using }\beta=\alpha_d\\
                &=w_0(\delta+\theta_d)                  &\text{using \Cref{wsontheta}}\\
                &=w_0(2\delta-\alpha_d)=2\delta+\beta   &\text{using }\beta=-w_0(\alpha_d)\\
                &=\alpha_d+2\delta=\tau_q(\alpha_d)     &\text{using }\beta=\alpha_d\text{ and \Cref{work2}}
        \end{align*}
    \item[{\bf Case 2}] Suppose $\beta\neq\alpha_d$. 
        Then $\beta,\iota(\beta)\in\J$, and $\varpi_\beta^\vee(\alpha_d)=\varpi_d^\vee(\beta)=0$.
        It follows from \Cref{defn:iota} that $\iota(\beta)=-\ws(\beta)$, hence $\ws(\iota(\beta))=-\beta$. 
        We have:\begin{align*}\hspace{20pt}
            w_0\ws w_d\ws(\alpha_d) &=w_0\ws(\delta+\iota(\beta))       &\qquad\text{using \Cref{work5}}\\
                                    &=w_0(\delta-\beta)=\delta-w_0(\beta)\\
                                    &=\delta+\alpha_d=\tau_q(\alpha_d)  &\qquad\text{using \Cref{work2}}\\
            w_0\ws w_d\ws(\beta)    &=w_0\ws w_d(-\iota(\beta))         &\qquad\\
                                    &=w_0\ws w_dw_d(\alpha_0)           &\qquad\text{using \Cref{work4}}\\
                                    &=w_0\ws(\alpha_0)=\delta+\beta     &\qquad\text{using \Cref{work6}}\\
                                    &=\tau_q(\beta)                     &\qquad\text{using \Cref{work2}}
        \end{align*}
\end{itemize}
Finally, we prove that \Cref{work1} holds for any $\alpha\in\dynk_0\backslash\left\{\alpha_d,\beta\right\}=\J\backslash\left\{\beta\right\}$.
Since $\alpha\in\dynk_0$, we have $-w_0(\alpha)\in\dynk_0$.
Observe further that since $\alpha\neq\beta$, we have $-w_0(\alpha)\neq\alpha_d$, and so $-w_0(\alpha)\in\J$.
Applying \Cref{defn:iota}, we get\begin{align}
        &\iota(-w_0(\alpha))=\ws w_0(\alpha)\nonumber\\
\implies&\ws\iota w_0(\alpha)=-w_0(\alpha)\label{work3}
\end{align}
We see from \Cref{work2} that $\tau_q(\alpha)=\alpha$.
We compute:\begin{align*}
    w_0\ws w_d\ws(\alpha)   &=-w_0\ws{}^\iota w_0\iota(\alpha)&\text{using }\alpha\not\in\J,\,\text{and \Cref{defn:iota}}\\
    &=-w_0\ws\iota w_0(\alpha)              &\text{using \Cref{iotaConj}}\\
    &=w_0^2(\alpha)=\alpha=\tau_q(\alpha)   &\text{using \Cref{work3}}
\end{align*}
\end{proof}

\begin{lemma}
\label{nilp}
Let $\Psi\define\roots-.d.\backslash\roots.\J.$, and consider some $\gamma\in\dynk_0\cup\pm\J$.
Then \begin{enumerate}
    \item All subsets of $\Psi$ are nilpotent (see \Cref{nilpotentRoots}).
    \item The set $\Psi\cup\{\gamma\}$ is nilpotent.
\end{enumerate}
\end{lemma}
\begin{proof}
First observe that 
$    \Psi=\roots-.d.\backslash\roots.\J.=\left\{\alpha\in\mathbb Z\dynk_d\mid-\theta_d\leq\alpha\leq-\alpha_0\right\}$.
Consider $\alpha,\beta\in\Psi$. 
Then $Supp(\alpha+\beta)\subset\dynk_d$.
Further, we have \begin{align*}
    \alpha,\beta\leq-\alpha_0\implies\alpha+\beta\leq-2\alpha_0.
\end{align*}
Now, since $\alpha_0$ is cominuscule in $\dynk_d$, it follows that $\alpha+\beta\not\in\roots..$.
Consequently, every subset of $\Psi$ is closed. 
Next, we prove that $\Psi\cup\left\{\gamma\right\}$ is closed.
\begin{enumerate}[leftmargin=*]
    \item   Suppose $\gamma=\alpha_d$.
        Consider $\alpha\in\Psi$. 
        The coefficient of $\alpha_0$ in $\alpha+\gamma$ is $-1$, and the coefficient of $\alpha_d$ is $1$.
        Hence, $\alpha+\gamma$ is not a root.
    \item   Suppose $\gamma\in\J$.
        Suppose further that $\alpha+\gamma\in\roots..$ for some $\alpha\in\Psi$. 
        Since $Supp(\alpha+\gamma)\subset\dynk_d$, we see that $\alpha+\gamma\in\roots.d.$. 
        Further, since the coefficient of $\alpha_0$ in $\alpha+\gamma$ is $-1$, we have $\alpha+\gamma\not\in\roots.\J.$. 
    \item   Suppose $\gamma\in-\J$. 
        Suppose further that $\alpha+\gamma\in\roots..$ for some $\alpha\in\Psi$. 
        Then $\alpha+\gamma\leq-\alpha_0$ and $\alpha+\gamma\in\roots.d.$.
        It follows that $\alpha+\gamma\in\Psi$.
\end{enumerate}
Finally, consider $u_+,u_-\in W$ given by\begin{align*}
    u_+=\begin{cases} w_d   &\text{if }\gamma=\alpha_d\\
            w_ds_\gamma     &\text{if }\gamma\in\J\\
            w_d             &\text{if }\gamma\in-\J  
        \end{cases}
&&  u_-=\begin{cases} s_{\alpha_d}  &\text{if }\gamma=\alpha_d\\
            s_\gamma                &\text{if }\gamma\in\J\\
            1                       &\text{if }\gamma\in-\J  
        \end{cases}
\end{align*}
It is easy to verify that $u_\pm\left(\Psi\cup\left\{\gamma\right\}\right)\subset\roots\pm..$.
\end{proof}

Recall from \Cref{defgpsi} the Lie sub-algebra $\lgl^\Psi\define\sum\limits_{\alpha\in\Psi}\lgl^\alpha$. 
Recall also from \Cref{titsFunctor} the isomorphism $\eta:\lgl^\Psi\rightarrow U_\Psi$ for some ordering on $\Psi$. 
\begin{prop}
\label{fibreId}
There exists a \para-equivariant isomorphism $\phi:\u_P\rightarrow X^-_\J(w_d^\J)$ given by $$\phi(X)=\eta(t^{-1}X)\,\mod\Para$$
where $X^-_\J(w_d^\J)$ is the opposite cell defined in \Cref{oppositeCell}.
\end{prop}
\begin{proof}
Observe that $\u_P=\lgl^{\roots+.0.\backslash\roots.\J.}$.
The map $X\mapsto t^{-1}X$ is \gl-equivariant (hence also \para-equivariant), and takes the root space $\alpha$ to $\alpha-\delta$.
Now, since\begin{align}
\label{minusdelta}
    \roots+.0.\backslash\roots.\J.-\delta &=\left\{\alpha-\delta\mid\alpha\in\roots.0.\backslash\roots.\J.\right\}\\ 
        &=\left\{\alpha-\delta\mid\alpha\in\mathbb Z\dynk,\,\alpha_d\leq\alpha\leq\theta_0\right\}\nonumber\\
        &=\left\{\alpha\in\mathbb Z\dynk\mid-\theta_d\leq\alpha\leq-\alpha_0\right\}=\roots-.d.\backslash\roots.\J.\nonumber
\end{align}
we have a map $t^{-1}:\u_P\rightarrow\lgl^\Psi$, where $\Psi\define\roots-.d.\backslash\roots.\J.$.
It follows from \Cref{oppositeCell} that $\eta\circ t^{-1}$ is an isomorphism from $\u_P$ to $Y_\J(w_d^\J)$.
It remains to show that $\eta$ is $P$-equivariant.
For $\gamma\in\dynk_0\cup-\J$, it follows from \Cref{nilp} that the group scheme $U_{\Psi\cup\{\gamma\}}$ is well-defined.
The action of $U_\gamma$ on $U_\Psi$ (resp. $\lgl^\Psi$) being the restriction of the adjoint action of $U_{\Psi\cup\{\gamma\}}$ on itself (resp. $\lgl^{\Psi\cup\{\gamma\}}$), the map $\eta$ is $U_\gamma$-equivariant.
It follows that $\eta$ is $P$-equivariant, since $P=\left\langle T,\,U_\gamma\mid\gamma\in\dynk_0\cup-\J\right\rangle$, and $\eta$ is $T$-equivariant by construction.
\end{proof}

\begin{theorem}
\label{form:phip}
Recall that $\u_P$ is the cotangent space of $G/P$ at identity. 
Let $q=w_0(\varpi_d^\vee)-\varpi_d^\vee$ as in \Cref{defn:q}.
The map $\phi:\u_P\rightarrow X^-_\J(w_d^\J)$ extends to a \gl-equivariant isomorphism $\phi:T^*\gl/\para\rightarrow Y_\J(\tau_q)$ (see \Cref{bigLGcell}) given by \begin{align*}
\hspace{80pt}   \phi(g,X)=g\,\phi(X)\,\mod\Para&&\text{for }g\in\gl,\ X\in\u_\para
\end{align*}
Let $\boldsymbol\theta:T^*G/P\rightarrow\lgl$ be the map given by $(g,X)\mapsto Ad(g)X$, and $\Ni=\operatorname{Im}(\boldsymbol\theta)$.
Let $\operatorname{pr}$ be the restriction of the quotient map $\gL/\Para\rightarrow\gL/\gO$ to $X^-_\J(\tau_q)$.
There exists an isomorphism $\Ni\rightarrow X^-_{\dynk_0}(\tau_q)$ such that the following diagram commutes:
\begin{center}
\begin{tikzcd}
    T^*\gl/\para\arrow[hookrightarrow, r, "\phi"] \arrow[d, twoheadrightarrow, "\boldsymbol\theta"]  & Y_\J(\tau_q)\arrow[d, twoheadrightarrow, "\operatorname{pr}"] \\
    \Ni\arrow[hookrightarrow, r, "\psi"]    & X^-_{\dynk_0}(\tau_q)
\end{tikzcd}
\end{center}
\end{theorem}
\begin{proof}
The proof of \cite[Theorem 1.3]{lakshmibai2015cotangent} applies to show that $\phi$ is well-defined and \gl-equivariant, and\begin{align*}
\hspace{60pt}   \phi(T^*G/P)\subset X_\J(w_0^\J w_d^\J)=X_\J(\tau_q)&&\qquad\text{using \Cref{defn:q}}
\end{align*}
Further, since $\phi$ is \gl-equivariant and $\phi(1,\u_P)=X^-_\J(w_d^\J)$, we have
    $$\phi(T^*G/P)=GX_\J^-(w_d^\J)=G\Borel^-X_J^-(w_d^\J)=Y_\J(\tau_q).$$ 
The isomorphism $\psi$ is from \cite{achar2013geometric,achar2012geometric} (in particular, see Proposition 6.6 of \cite{achar2012geometric}).
We verify that for $X\in\u_P$, we have $$\operatorname{pr}(\phi(1,X))=\phi(X)\,\mod\gO=\psi(\boldsymbol\theta(X)).$$
Since $\phi$, $\psi$ are \gl-equivariant isomorphism (\emph{loc. cit.}), the diagram commutes.
\end{proof}

%% file: good.tex
\section{The Conormal Variety of a Schubert variety}
\label{good}
Let $\dynk_0$ be a finite type Dynkin diagram and \dynk\ the associated extended diagram.
Fix a cominuscule root $\alpha_d\in\dynk_0$ and let $\J,\iota,\gl,\gL,\para,\Para$, and $\phi:T^*G/P\hookrightarrow\gL/\Para$ be as in the previous section.
We fix $w\in W_0^\J\define W_0\cap W^\J$ and set\begin{align}\label{defn:v}
    v\define{}^\iota(w_0w\ws).
\end{align}
Let \con\ be the conormal variety of $X_\J(w)$ in $G/P$.
It follows from \Cref{form:phip} that the closure of $\phi(\con)$ in $X_\J(\tau_q)$ is a $B$-stable compactification of \con.
The primary goal of this section is \Cref{mainResult}:
This compactification of $\phi(\con)$ is a Schubert subvariety of $X_\J(\tau_q)$ if and only if $X_\J(w_0w\ws)$ is smooth.
This yields powerful results (\Cref{main1,main2}) regarding the geometry of \con\ in the case where $X_\J(w_0w\ws)$ is smooth. 
\par 
\par 
We first introduce the notion of the conormal variety and give a combinatorial description of \con\ in \Cref{prop:uw}.  
We then develop a sequence of lemmas leading to the proof of \Cref{mainResult}.
Finally, we show in \Cref{conFibreId} that the fibre at identity of \con\ can be identified as the union of opposite cells of certain Schubert varieties. 

\subsection{The Conormal Variety}
Consider the Schubert variety $X_\J(w)$, viewed as a subvariety of $G/P$.
For $x$ a smooth point in $X_\J(w)$, the \emph{conormal fibre} $N^*_xX_\Para(w)$ is the annihilator of $T_xX_\J(w)$ in $T^*_xG/P$, i.e., for $x$ a \emph{smooth} point,
\begin{align*}
    N^*_xX_\Para(w)=\left\{f\in T^*_xG/P\mid f(v)=0\ \forall v\in T_xX_\Para(w)\right\}
\end{align*}
The \emph{conormal variety} \con\ of $X_\J(w)\hookrightarrow\gl/\para$ is the closure in $T^*\gl/\para$ of the conormal bundle of the smooth locus of $X_\J(w)$.

\begin{prop}
\label{prop:uw}
Let $\rw\define\left\{\alpha\in\proots\mid\alpha\geq\alpha_d,\,w(\alpha)>0\right\}$ and $\lgl^R=\sum\limits_{\alpha\in R}\lgl^\alpha$.
The conormal variety \con\ is the closure in $T^*\gl/\para$ of 
$$  \left\{(bw,X)\in\gl\times^\para\u_\para\mid b\in\borel,\,X\in\lgl^R\right\}$$
\end{prop}
\begin{proof}
The tangent space of $G/P$ at identity is $\lgl/\mathfrak p$. 
Consider the action of $P$ on $\lgl/\mathfrak p$ induced from the adjoint action of $P$ on $\lgl$.
The tangent bundle $T\,\gl/\para$ is the fibre bundle over $\gl/\para$ associated to the principal \para-bundle $\gl\rightarrow\gl/\para$, for the aforementioned action of \para\ on $\lgl/\mathfrak p$
, i.e., $T\,\gl/\para=\gl\times^\para\lgl/\mathfrak p$.
Let \begin{align*}
    R'  &=\roots+..\backslash\left(\roots+.\J.\cup R\right)=\left\{\alpha\in\roots+..\mid\alpha\geq\alpha_d,\,w(\alpha)<0\right\}\\
    U_w &=\left\langle U_\alpha\mid-\alpha\in R'\right\rangle
\end{align*}
For any point $b\in B$, we have (see, for example \cite{borel2012linear}):\begin{align*}
    BwP\ (\mod P)   &=bBwP\ (\mod P)\\
                    &=b(wU_ww^{-1})wP\ (\mod P)\\
                    &=bwU_wP\ (\mod P)
\end{align*}
It follows that the tangent subspace at $bw$ of the big cell $BwP\ (\mod P)$ is given by
\begin{align*}
    T_wBwP\ (\mod P)&=\left\{(bw,X)\in G\times^P\lgl/\mathfrak p\,\mid\, X\in\bigoplus\limits_{-\alpha\in R'}\lgl^\alpha/\mathfrak p\right\}
\end{align*}
where $\lgl^\alpha/\mathfrak p$ denotes the image of a root space $\lgl^\alpha$ under the map $\lgl\rightarrow\lgl/\mathfrak p$.
\par 
\par 
Recall that the Killing form identifies the dual of a root space $\lgl^\alpha$ with the root space $\lgl^{-\alpha}$. 
Consequently, a root space $\lgl^\alpha\subset\u$ annihilates $T_{bw}BwP\,(\mod P)$ if and only if $\alpha\in\proots\backslash\roots+.\J.$ and $\alpha\not\in R'$,
or equivalently, $\alpha\in R$.
The result now follows from the observation that $BwP\,(\mod P)$ is a dense open subset of $X_P(w)$, and is contained in the smooth locus of $X_P(w)$.
\end{proof}

\begin{lemma}
\label{lem:vInWSd}
Recall from \Cref{defn:v} that $v={}^\iota(w_0w\ws)$.
We have:\begin{enumerate}
    \item $W_0\cap W^\J=W_0^d\define W_0\cap W^d$ and $W_d\cap W^\J=W_d^0\define W_d\cap W^0$.
    \item $v\in W_d^0$.
    \item $l(wv)=l(w)+l(v)=\dim\gl/\para$.
\end{enumerate}
\end{lemma}
\begin{proof}
It follows from \Cref{support} that $W_0\subset W^{\{\alpha_0\}}$, hence\begin{align*}
    W_0\cap W^\J\subset W_0\cap W^{\{\alpha_0\}}\cap W^\J=W_0^d.
\end{align*}
Conversely, since $W^d\subset W^\J$, we have $ 
    W_0^d=W_0\cap W^d\subset W_0\cap W^\J
$. 
Consequently, $$W_0\cap W^\J=W_0^d.$$
Applying $\iota$ to this equality, we have $W_d\cap W^\J=W_d^0$.
\par 
\par 
Next, we prove $v\in W^\J$.
It is sufficient to verify $v(\roots+.\J.)\subset\roots+..$, see \Cref{eq:minRep}.
\begin{align}\hspace{60pt}
            v(\roots+.\J.)  &=\iota w_0w\ws(\iota(\roots+.\J.)) &\text{using \Cref{iotaConj}}\nonumber\\
\implies    v(\roots+.\J.)  &=\iota w_0w\ws(\roots+.\J.)        &\text{using \Cref{actionW}}\nonumber\\
\implies    v(\roots+.\J.)  &=\iota w_0w(\roots-.\J.)           &\text{using \Cref{weylInv}}\label{work10}
\end{align}
Now, since $Supp(w)\subset\dynk_0$, we see from \Cref{support} that $w(\roots-.\J.)\subset\roots.0.$.
Further, since $w\in W^\J$, it follows from \Cref{eq:minRep} that $w(\roots-.\J.)\subset\roots-.0.$.
Applying \Cref{work10}, we have $
     v(\roots+.\J.) \subset\iota w_0(\roots-.0.)=\iota(\roots+.0.)\subset\roots+..
$. 
This proves $v\in W^\J$.
\par 
\par 
Further, since $w_0,w,\ws\in W_0$, we have $w_0w\ws\in W_0$. 
It follows from \Cref{actionW} that $v={}^\iota(w_0w\ws)\in W_d$.
Combining with $v\in W^\J$, we get $(2)$: $$v\in W^\J\cap W_d=W_d^0$$
Finally, since $v\in W^\J$, we have ${}^\iota v=w_0w\ws\in W^\J$. 
Consequently,\begin{align*}
    l(w_0w)         &=l(w_0w\ws)+l(\ws) \\
\implies\dim\gl-l(w)    &=l(v)+\dim\para    \\
\implies\hspace{18pt}\dim\gl/\para   &=l(w)+l(v)=l(wv)
\end{align*}    
where the last equality follows from the observations $w\in W_0$ and $v\in W^0$.
\end{proof}

\begin{lemma}
Let $u\in W_d^0$. 
Then $Supp(u)$ is a connected sub-graph of $\dynk_d$.
\end{lemma}
\begin{proof}
Suppose $Supp(u)$ is not connected.
Let $\L_1$ be the connected component of $Supp(u)$ containing $\alpha_d$, and let $\L_2\define Supp(u)\backslash\L_1$.
Now, since $\L_1$ and $\L_2$ are disconnected, we have\begin{align*}\hspace{90pt}
    s_\alpha s_\beta=s_\beta s_\alpha&&\forall\,s_\alpha\in\L_1,\,s_\beta\in\L_2.
\end{align*}
Let $s_1\ldots s_l$ be a reduced word for $u$, and let $k$ be the largest index such that $s_k\in\L_2$.
Then $s_ks_m=s_ms_k$ for all $m>k$.
It follows that \begin{align*}
    us_k=(s_1\ldots s_l)s_k &=(s_1\ldots s_{k-1}s_{k+1}\ldots s_ls_k)s_k\\
                            &=s_1\ldots s_{k-1}s_{k+1}\ldots s_l
\end{align*}
Now, since $Supp(u)\subset\dynk_d$, we have $\alpha_0,\alpha_d\not\in\L_2$.
In particular, $s_k\in W_\J$. Hence $us_k=u\ (\mod W_\J)$ and $us_k<u$, contradicting the assumption $u\in W^0_d\subset W^\J$.
\end{proof}

\begin{prop}
\label{sb}
\label{veeLemma}
For $u\in W^0_d$, the following are equivalent:\begin{enumerate}
    \item $X_\J(u)$ is smooth. 
    \item $X_\J(u)$ is $\Para_\L$-homogeneous, i.e., $X_\J(u)=\Para_\L/(\Para_\L\cap\Para)$ for some connected sub-graph $\mathcal L\subset\dynk_d$.
    \item $l(u^{-1}\star u\ws)=l(u\ws)$, where $\star$ is the Demazure product, see \Cref{lem:bruProd}.
    \item $(u\ws)^{-1}(\alpha)<0$ for all $\alpha\in Supp(u)$.
    \item $\left\{\alpha\in\roots+..\mid u(\alpha)<0\right\}=\roots+.Supp(u).\backslash\roots+.\J.$.
    \item $u=w_{\mathcal L}w_{\mathcal L\cap\J}$, where $\mathcal L=Supp(u)$, and $w_\L$, $w_{\mathcal L\cap\J}$ denote the maximal elements in $W_\L$, $W_{\mathcal L\cap\J}$ respectively.
\end{enumerate}
\end{prop}
\begin{proof}
The claim $(1)\iff(2)$ is \cite[Theorem 1.1]{billey2010smooth}.
Suppose $(2)$ holds, i.e., $X_\J(u)=\Para_\L/\Para_{\L\cap\J}$.
We have the following Cartesian square:\[
\begin{tikzcd}
                        &[-2em]X_\Borel(u\ws)\arrow[d]\arrow[r,hook]&\gL/\Borel\arrow[d]\\
    \Para_\L/\Para_{\L\cap\J}\arrow[r,equal]&X_\J(u)\arrow[r,hook]  &\gL/\Para
\end{tikzcd}
\]
Since $X_\J(u)$ is $\Para_\L$-stable, the same is true of its pull-back $X_\Borel(u\ws)$.
Further, since $u\in W_\L$, any lift of $u^{-1}$ to $N(\mathcal K)$ (see \Cref{def:kmg}) is in $\Para_\L$.
Consequently, $X_\Borel(u\ws)$ is $u^{-1}$ stable, and so $u^{-1}\star u\ws=u\ws$.
Hence we obtain $(2)\implies (3)$.
\par 
\par 
It is clear from \Cref{formula:bruProd} that $(3)$ is equivalent to $u^{-1}\star u\ws=u\ws$, which holds if and only if $s_\alpha \star u\ws=u\ws$ for all $\alpha\in\mathcal L$. 
This is equivalent to $(4)$ from \Cref{lem:minRule}. 
Hence we obtain $(3)\iff(4)$.
\par 
\par 
Suppose $(4)$ holds.
Then $(u\ws)^{-1}(\alpha)<0$ for all $\alpha\in\roots+.Supp(u).$.
It follows from \Cref{support} and \Cref{eq:minRep} that\[
    \left\{\alpha\in\roots+..\mid u(\alpha)<0\right\}\subset\roots+.Supp(u).\backslash\roots+.\J.
\]
Consider $\alpha\in\roots+.Supp(u).\backslash\roots+.\J.$ satisfying $u(\alpha)>0$. 
Since $u$ preserves \roots.Supp(u)., we have $u(\alpha)\in\roots+.Supp(u).$.
Applying $(4)$ to $u(\alpha)$, we get $\ws u^{-1}(u(\alpha))=\ws(\alpha)<0$. 
It follows that $\alpha\in\roots+.\J.$, contradicting the assumption $\alpha\not\in\roots+.\J.$.
Hence we obtain the implication $(4)\implies(5)$.
\par 
\par 
Suppose $(5)$ holds.
We verify that 
\begin{align*}
    \left\{\alpha\in\roots+..\mid w_\L w_{\L\cap\J}(\alpha)<0\right\}=\roots+.\mathcal L.\backslash\roots+.\J.
\end{align*}
Now since $u\in W$ is uniquely determined by the set $\left\{\alpha\in\roots+..\mid u(\alpha)<0\right\}$ (cf. \cite[\S{1.3.14}]{kumar2012kac}), we get $u=w_\L w_{\L\cap\J}$. 
Hence we obtain $(5)\implies(6)$.
\par 
\par 
Finally, suppose $(6)$ holds. 
Since $w_{\L\cap\J}\subset\J$, we have $u=w_\L\ (\mod W_\J)$. 
It follows that $X_\J(u)=X_\J(w_\L)=\Para_\L/(\Para_{\L\cap\J})$.
Hence we obtain $(6)\implies(2)$.
\end{proof}

\begin{lemma}
\label{calc1}
For $\alpha\in\roots+.0.\backslash\roots.\J.$, we have $v(\alpha-\delta)=-\iota w_0w(\alpha)$.
\end{lemma}
\begin{proof}
Since $\alpha\in\roots+.0.\backslash\roots.\J.$, we have $\alpha\geq\alpha_d$. 
Set $\gamma=\alpha-\alpha_d$.
Since $\alpha_d$ is cominuscule, $Supp(\gamma)\subset\J$.
In particular, $\iota(\gamma)=-\ws(\gamma)$.
We compute:\begin{align}
    \iota(\alpha)           &=\iota(\alpha_d)+\iota(\gamma)=\alpha_0-\ws(\gamma)\nonumber\\
\implies\iota(\alpha-\delta)&=-\theta_0-\ws(\gamma)\label{work12}
\end{align}
Recall from \Cref{defn:v} that $v={}^\iota(w_0w\ws)$.
We now compute $v(\alpha-\delta)$:
\begin{align*}\hspace{10pt}
    v(\alpha-\delta)&={}^\iota(w_0w\ws)(\alpha-\delta)
                     =\iota w_0w\ws(\iota(\alpha-\delta))       &\text{using \Cref{actionW}}\\
                    &=-\iota w_0w\ws(\theta_0+\ws(\gamma))      &\text{using \Cref{work12}}\\
                    &=-\iota w_0w(\alpha_d)-\iota w_0w(\gamma)  &\text{using \Cref{wsontheta}}\\
                    &=-\iota w_0w(\alpha_d+\gamma)=-\iota w_0w(\alpha)
\end{align*}
\vskip -10pt
\end{proof}

\begin{lemma}
\label{lem:involution}
The map $\alpha\mapsto\alpha-\delta$ induces a bijection\begin{align*}
    \left\{\alpha\in\proots\mid\alpha\geq\alpha_d,\,w(\alpha)>0\right\} &\xrightarrow{\scriptscriptstyle\sim}\left\{\alpha\in\roots-.d.\mid v(\alpha)>0\right\}
\end{align*}
\end{lemma}
\begin{proof}
Observe that \begin{align*}\hspace{30pt}
    -\iota w_0(\roots\pm.0.)=\iota(-w_0\roots\pm.0.)&=\iota(\roots\pm.0.) &\text{using \Cref{weylInv}}\\
                                                    &=\roots\pm.d.        &\text{using \Cref{actionW}}
\end{align*}
Now, since $v(\alpha-\delta)=-\iota w_0w(\alpha)$ (see \Cref{calc1}), it follows that for $\alpha\in\roots+.0.\backslash\roots.\J.$, $w(\alpha)>0$ is equivalent to $v(\alpha-\delta)>0$.
The result now follows from \Cref{minusdelta}, which states that $\alpha\mapsto\alpha-\delta$ induces a bijection from $\roots+.0.\backslash\roots.\J.$ to $\roots-.d.\backslash\roots.\J.$.
\end{proof}

\begin{prop}
Recall the map $\phi$ from \Cref{fibreId}.
Let $$\rw=\left\{\alpha\in\proots\mid\alpha\geq\alpha_d,\,w(\alpha)>0\right\}$$
Then $\phi(\lgl^{\rw})$ is dense in $v^{-1}X_\J(v)$.
\end{prop}
\begin{proof}
Following the proof of \Cref{fibreId}, we see that $\eta(t^{-1}\lgl^R)=U_\Phi$, where $\Phi\define\left\{\alpha-\delta\mid\alpha\in R\right\}$.
It follows from \Cref{lem:involution} that 
$$\Phi=\left\{\alpha\in\roots-.d.\mid v(\alpha)>0\right\}=v^{-1}(\roots+..)\cap\roots-..$$
In particular, $\#\Phi=l(v)$ (see \cite[\S1.3.14]{kumar2012kac}), and so 
$$\dim\lgl^R=\dim U_\Phi=\#\Phi=l(v)=\dim X_\J(v)=\dim v^{-1} X_\J(v)$$ 
Further observe that
$$\phi(\lgl^\rw)=\eta(t^{-1}\lgl^R)\ (\mod\Para)\subset v^{-1}\Borel v\ (\mod\Para)\subset v^{-1}X_\J(v)$$
The result follows from the injectivity of $\phi$ and the irreducibility of $v^{-1} X_\J(v)$.
\end{proof}
\begin{theorem}
\label{mainResult}
The closure of $\phi(\con)$ in $\gL/\Para$ is a Schubert variety if and only if $X_\J(w_0w\ws)$ is smooth.
\end{theorem}
\begin{proof}
For $(bw,X)$ be a generic point in \con, we have \begin{align*}
    \phi(bw,X)=bw\,\phi(X)\in\Borel wv^{-1}\Borel v\Para\ (\mod\Para) 
\end{align*}
Hence the minimal Schubert variety containing $\phi(\con)$ is $X_\J(wv^{-1}\star v)$.
Consequently, the closure $\overline{\phi(\con)}$ is a Schubert variety if and only if\begin{align}\label{work20}
    dim X_\J(wv^{-1}\star v)=\dim\con=\dim\gl/\para
\end{align}
Consider the following Cartesian diagram:\[
\begin{tikzcd}
    X_\Borel(wv^{-1}\star v\star\ws)\arrow[d]\arrow[r,hook] &\gL/\Borel\arrow[d]\\
    X_\J(wv^{-1}\star v)\arrow[r,hook]                      &\gL/\Para
\end{tikzcd}
\]
The dimension of the generic fibre for the right projection is $\dim\Para/\Borel$.
Observe that $X_\Borel(wv^{-1}\star v\star\ws)$ is the pullback of $X_\J(wv^{-1}\star v)$ to $\gL/\Borel$. 
It follows that \Cref{work20} is equivalent to \begin{align}\label{work21}
    \dim X_\Borel(wv^{-1}\star v\star\ws)=\dim\gl/\para+\dim\Para/\Borel=\dim\gl/\borel
\end{align}
We see from \Cref{lem:vInWSd} and \Cref{lengthVee} that $wv^{-1}=w\star v^{-1}$ and $v\star\ws=v\ws$. 
Hence, we have:\begin{align}\label{work22}
    wv^{-1}\star v\star\ws=w\star v^{-1}\star v\ws
\end{align}
Observe that since $v,\ws\in W_d$, we have $v^{-1}\star v\ws\in W_d$.
Recall also from \Cref{lem:vInWSd} that $w\in W^\J\cap W_0\subset W^d$.
It follows that\begin{align}
        w\star v^{-1}\star v\ws     &=w(v^{-1}\star v\ws)   &\text{using \Cref{lengthVee}}\nonumber\\
\implies l(wv^{-1}\star v\star\ws)  &=l(w(v^{-1}\star v\ws))&\text{using \Cref{work22}}\nonumber\\
\label{work23}                  &=l(w)+l(v^{-1}\star v\ws)  &\text{using }w\in W^d,\,v^{-1}\star v\ws\in W_d
\end{align}
Now $l(v^{-1}\star v\ws)\geq l(v\ws)$, with equality holding if and only if $v^{-1}\star v\ws=v\ws$. 
Continuing \Cref{work23}, we have \begin{align*}
    \dim X_\Borel(wv^{-1}\star v\star\ws)   &\geq l(w)+l(v\ws)\\ 
                            &=l(w)+l(v)+l(\ws)          &\text{}v\in W^\J,\,\ws\in W_\J\\
                            &=l(wv)+l(\ws)              &\text{}w\in W^d,\,v\in W_d\\
                            &=\dim G/B
\end{align*}
Hence, \Cref{work21} holds if and only if $v^{-1}\star v\ws=v\ws$, which is equivalent to $X_\J(v)$ being smooth, see \Cref{veeLemma}.
Finally, the Schubert varieties $X_\J(v)$ and $X_\J(w_0w\ws)$ being isomorphic (since $v={}^\iota(w_0w\ws)$), we deduce that the closure $\overline{\phi(\con)}$ is a Schubert variety if and only if $X_\J(w_0w\ws)$ is smooth.
\end{proof}

\begin{theorem}
\label{main1}
Let $w\in W^\J_0$ be such that $X_\J(w_0w\ws)$ is smooth. 
Then \con\ is normal, Cohen-Macaulay, and has a resolution via Bott-Samelson varieties.
Further, the family $\left\{\con\mid X_\J(w_0w\ws)\text{ is smooth}\right\}$ is compatibly Frobenius split.
\end{theorem}
\begin{proof}
These are standard results for Schubert varieties.
One can find details in \cite{faltings2003algebraic,littelmann2003bases,mehta1985frobenius}. 
\end{proof}

\begin{theorem}
\label{main2}
Let $w\in W^\J_0$ be such that $X_\J(w_0w\ws)$ is smooth. 
Let $V(\lambda)$ be the simple module associated to a dominant weight $\lambda$ corresponding to \Para. 
Recall the monomial basis $\mathbb M(\lambda)$ and the associated elements $u_\pi\in V(\lambda)$ as developed in \cite{littelmann2003bases}.
The ideal sheaf of $\con$ in $T^*G/P$ is $\phi^{-1}\mathcal I$, where $\mathcal I=\left\langle u_\pi\mid\pi\leq\tau_q,\,\pi\not\leq wv\right\rangle$.
\end{theorem}
\begin{proof}
Recall (cf. \cite{littelmann2003bases}) that $\left\{u_\pi\mid \pi\leq\tau_q\right\}$ is a basis for $H^0(X_\J(\tau_q),L(\lambda))$, where $L(\lambda)$ is the line bundle associated to $\lambda$.
Further, $\mathcal I=\left\langle u_\pi\mid\pi\leq\tau_q,\,\pi\not\leq wv\right\rangle$ is the ideal sheaf of $X_\J(wv)$ in $X_\J(\tau_q)$.
Since \con\ is closed in $T^*G/P$, we have \begin{align*}
        \phi(\con)  &=\overline{\phi(\con)}\cap\phi(T^*G/P) \\ 
                &=X_\J(wv)\cap L^-\gl\boldsymbol e\cap X_\J(\tau_q)=Y_\J(wv)
\end{align*}
It follows that the ideal sheaf of \con\ in $T^*G/P$ is the pull-back (via $\phi$) of the restriction of $\mathcal I$ to $Y_\J(\tau_q)$, i.e., the ideal sheaf is $\phi^{-1}(\mathcal I)$.
\end{proof}

\begin{prop}
\label{conFibreId}
Let $w\in W_0^\J$ be such that $X_\J(w_0w\ws)$ is smooth, and let $N^*_0X_\J(w)$ denote the fibre at identity of the conormal variety \con.
Then \begin{align*}
    \phi(N^*_0X_\J(w))=\bigcup\limits_{u\in\mathcal S} X^-_\J(u)
\end{align*}
where $\mathcal S=\left\{u\in W_d^0\mid u\leq (wv)^{\dynk_0}\right\}$, and $(wv)^{\dynk_0}$ is the minimal representative of $wv$ with respect to $\dynk_0$.
\end{prop}
\begin{proof}
Recall from \Cref{fibreId} that $\phi(T_{\boldsymbol e}^*G/P)=X^-_\J(w_d^\J)$. 
It follows that \begin{align*}
    \phi(N^*_0X_\J(w))=X^-_\J(w_d^\J)\cap X_\J(wv)=\bigcup X^-_\J(u) 
\end{align*}
where the union runs over $\left\{u\in W^\J\mid u\leq w_d^\J,\,u\leq wv\right\}$.
Since $w_d^\J$ is maximal in $W_d^\J$, the condition $\left\{u\in W^\J,\ u\leq w_d^\J\right\}$ is equivalent to $u\in W_d\cap W^\J=W_d^0$, see \Cref{lem:vInWSd}. 
Hence,\[
    \left\{u\in W^\J\mid u\leq w_d^\J,\,u\leq wv\right\}=\left\{u\in W^0_d\mid u\leq w_d^\J,\,u\leq wv\right\}
\]
Finally, consider $u\in W^0$.
If $u\leq wv$, then $u\leq (wv)^{\dynk_0}$. 
It follows that \begin{align*}
    \left\{u\in W^\J\mid u\leq w_d^\J,\,u\leq wv\right\}&=\left\{u\in W^0_d\mid u\leq w_d^\J,\,u\leq wv\right\}\\
        &=\left\{u\in W_d^0\mid u\leq(wv)^{\dynk_0}\right\}=\mathcal S
\end{align*}
\vskip-4pt
\end{proof}

%% file: determinantal.tex
\setlength\parskip{5pt}
\section{Determinantal Varieties}
\label{sec:det}
In this section, we use the results of \Cref{good} to prove the following:
The conormal fibre at the zero matrix of the rank $r$ (usual, symmetric, skew-symmetric resp.) determinantal variety is the co-rank $r$ (usual, symmetric, skew-symmetric resp.) determinantal variety.
\par 
\par 
Consider a rank $r$ (usual, symmetric, skew-symmetric resp.) determinantal variety $\Sigma$. 
There exists a simply connected, almost simple group \gl\ (of type $A$, $C$, $D$ resp.) and a cominuscule Grassmannian $G/P$ such that $\Sigma$ is naturally identified as the opposite cell of some Schubert variety $X_\J(w)\subset G/P$, see \cite{lakshmibai1978geometry}.
For such $w$, we verify that the Schubert variety $X_\J(w_0w\ws)$ is smooth.
This allows us to apply \Cref{conFibreId}. 
Finally, we show that the union of the various Schubert varieties in \Cref{conFibreId} is equal to a single Schubert variety, which we further verify to be isomorphic to the co-rank $r$ (usual, symmetric, skew-symmetric resp.) determinantal variety.
\par 
\par 
We carry out the proof in detail only for the skew-symmetric determinantal varieties.
The other two cases are completely analogous.
For the usual determinantal varieties, this result has been proved by Strickland \cite{strickland1982conormal}.
Further, the conormal fibres at the zero matrix of the usual determinantal varieties and symmetric determinantal varieties have been studied by Gaffney and Rangachev (cf. \cite{gaffney2014pairs}) 
and Gaffney and Lira (cf. \cite{mich}).
\subsection{The Weyl Group of $D_n$}
Let $\dynk_0=D_n$, $\dynk=\widetilde D_n$, and $W_0$ (resp. $W$) the Weyl group of $\dynk_0$ (resp. \dynk).
For $1\leq i,j\leq n-1$, the braid relations in $W_0$ are \begin{align*}
\hspace{90pt}    s_is_j      &=s_js_i    &\vert i-j\vert\geq 2\\
\hspace{90pt}    s_is_js_i   &=s_js_is_j &\vert i-j\vert=1
\end{align*}
The remaining braid relations are $s_ns_i=s_is_n$ for $i\neq n-2$ and $$s_ns_{n-2}s_n=s_{n-2}s_ns_{n-2}$$
Let $\mu$ be the involution on $\left\{1,\ldots,2n\right\}$ given by $\mu(i)\define 2n+1-i$.
We embed the Weyl group $W_0$ into the symmetric group $S_{2n}$ as follows\begin{align}\label{WeylEmbed}
    W_0=\left\{w\in S_{2n}\mid w\mu=\mu w,\,sgn(w)=1\right\}
\end{align}
via the homomorphism (cf. \cite{lakshmibai1978geometry}) given by
\begin{align*}
\hspace{90pt}   s_i &\mapsto r_ir_{2n-i}            &\quad i\neq n\\
\hspace{90pt}   s_n &\mapsto r_nr_{n-1}r_{n+1}r_n   & 
\end{align*}
where $r_i$ denotes the transposition $(i\ i+1)$ in $S_{2n}$.
It is clear that $w\in W_0$ is uniquely determined by its value on $1,\ldots,n$.
Accordingly, we represent $w$ by the string $[w(1),\ldots,w(n)]$.

\subsection{The Involution $\iota$ for $D_n$}
The simple root $\alpha_n$ is cominuscule in $\dynk_0=D_n$, and $\J\define\dynk_0\backslash\{\alpha_n\}=\left\{\alpha_1,\ldots,\alpha_{n-1}\right\}$ is isomorphic to $A_{n-1}$.
Let $\iota$ be the involution defined in \Cref{defn:iota}.
Recall that the action of the \emph{Weyl involution} $-w_\J$ on $\J\cong A_{n-1}$ is given by $-w_\J(\alpha_i)=\alpha_{n-i}$ (cf. \cite[Ch.VI\S4.7]{BourbakiNicolas2008Lgal}), and so $\iota(\alpha_i)=\alpha_{n-i}$ for $1\leq i\leq n-1$.
Further, since $\iota$ interchanges $\alpha_0$ and $\alpha_n$, we have
\begin{align}\label{specificiota}
\hspace{80pt}   \iota(\alpha_i)=\alpha_{n-i}&&\forall\,\alpha_i\in\dynk.
\end{align}

\subsection{Skew-Symmetric Determinantal Varieties}
\label{sec:ssdv}
Let $M^{sk}_n$ be the variety of skew-symmetric $n\times n$ matrices.
The \emph{rank $r$ skew-symmetric determinantal variety} $\overline\Sigma_r^{sk,n}$ is the subvariety of $M^{sk}_n$ given by:\begin{align*}
    \overline\Sigma_r^{sk,n}&=\left\{A\in M^{sk}_n\mid rank(A)\leq r\right\} 
\end{align*}
Recall that the rank of a skew-symmetric matrix is necessarily even.
Hence, we assume without loss of generality that \emph{$r$ is even}.
\par 
\par 
Let $G$ be the simply connected, almost-simple group of type $D_n$, and let $P\subset G$ be the parabolic group corresponding to $\J=\left\{\alpha_1,\ldots,\alpha_{n-1}\right\}$, see \Cref{TBL:comin}.
Following \cite{lakshmibai1978geometry}, we identify $M^{sk}_n$ with the opposite cell in $G/P$. 
Under this identification, the zero matrix corresponds to $\boldsymbol e\in G/P$, and $\overline\Sigma_r^{sk,n}=X_\J^-(w_r)$, where
\begin{align}\label{form:skew}
    w_r &\define[r+1,\ldots,n,2n-r+1,\ldots,2n]
\end{align}
in the sense of \Cref{WeylEmbed}.
Observe that \begin{align}\label{wrbelongs}
    w_r\in W^\J\cap W_0=W_0^{\J\cup\{\alpha_0\}}
\end{align}
The last equality is \Cref{lem:vInWSd}, (1). 

\begin{remark}
In \cite{lakshmibai1978geometry}, the skew-symmetric variety is identified with a Schubert variety corresponding to the group $SO(2n)$, which is not simply connected.
This is not a problem however, since Schubert varieties depend only on the underlying Dynkin diagram, and not on the group per se.
\end{remark}

\begin{theorem}
\label{fibreDet}
The conormal fibre of $\overline\Sigma_r^{sk,n}$ at $0$ is isomorphic to $\overline\Sigma_{\overline n-r}^{sk,n}$ where\begin{align*}
    \overline n=\begin{cases} n&\text{if $n$ is even,}\\n-1 &\text{if $n$ is odd.}\end{cases}
\end{align*}
\end{theorem}
\begin{proof}
Let $\dynk_0$ (resp. \dynk, resp. \J) be the Dynkin diagram $D_n$ (resp. $\widetilde D_n$, $D_n\backslash\{\alpha_n\}$), and $W_0$ (resp. $W$, $W_\J$) its Weyl group.
Recall that $\J\cong A_{n-1}$.
For $\mathcal L$ a sub-diagram of $\dynk_0$, we write $w_{\mathcal L}$ for the longest element in $W$ supported on $\mathcal L$, and $w_{\mathcal L}^\J$ for its minimal representative with respect to \J.
The longest elements $w_0\in W_0$ and $\ws\in W_\J$ are given by\begin{align*}
    w_0=w_{\overline n} &=[2n,\ldots,n+2,\overline n+1]\\
    \ws                 &=[n,\ldots,1]
\end{align*}
in the sense of \Cref{WeylEmbed}, see \cite{lakshmibai1978geometry}.
Let $w_r$ be as defined in \Cref{form:skew}, and set $v_r\define{}^\iota(w_0w_r\ws)$.
We have\begin{align}\label{wlj}
    w_0w_r\ws =[1,\ldots,r,\overline n+1,n+2,\ldots,2n-r] =w^\J_{\mathcal L}
\end{align}
where $\mathcal L=\left\{\alpha_{r+1},\ldots,\alpha_n\right\}$.
Hence $X_\J(w_0w_r\ws)$ is smooth, see \Cref{sb}. 
It now follows from \Cref{conFibreId} and \Cref{intersectw} that \begin{align*}
    \phi(N^*_0X_\J(w_r))=X^-_\J({}^\iota w_{\overline n-r})\cong X^-_\J(w_{\overline n-r})\cong\overline\Sigma_{\overline n-r}^{sk,n}.
\end{align*}
\end{proof}

It remains to prove \Cref{intersectw}.
The proof is obtained as a consequence of the following two lemmas.

\begin{lemma}
\label{rel1}
Consider $x_i\in W_0$ defined inductively as \begin{align}\label{rec:x}
\hspace{60pt}    x_i &=\begin{cases} s_n      &\text{for }i=n-1\\
                            s_{i+1}s_ix_{i+1} &\text{for }1\leq i<n-1
                    \end{cases}
\end{align}
Then $s_{i+2}s_{i+3}x_i=x_is_is_{i+1}$ for $1\leq i\leq n-4$.
\end{lemma}
\begin{proof}
We see from the braid relations that $s_jx_i=x_is_j$ for $j\leq i-2$.
In particular, $s_is_{i+1}x_{i+3}=x_{i+3}s_is_{i+1}$.
Now \begin{align*}\hspace{14pt}   
    s_{i+2}s_{i+3}x_i   
    &=s_{i+2}s_{i+3}s_{i+1}s_i\,s_{i+2}s_{i+1}\,s_{i+3}s_{i+2}x_{i+3}   &\qquad\text{using \Cref{rec:x}}\\
    &=s_{i+1}s_{i+2}s_{i}s_{i+1}s_{i+3}s_{i+2}\,s_{i}s_{i+1}x_{i+3}&\qquad\text{using Braid relations}\\
    &=s_{i+1}s_{i}s_{i+2}s_{i+1}s_{i+3}s_{i+2}\,x_{i+3}s_{i}s_{i+1}&\qquad\text{using Braid relations}\\
    &=x_{i}s_{i}s_{i+1}&\qquad\text{using \Cref{rec:x}}
\end{align*}
\end{proof}

\begin{lemma}
\label{rel2}
Let $x_i$ be given by \Cref{rec:x}. 
For $3\leq j, k\leq\overline n-1$, we have
\begin{align*}
    {}^\iota x_{\overline n-k}\,x_k                     &=x_{k-2}\,{}^\iota x_{\overline n-k+2}\\
    {}^\iota x_{\overline n-k}\,x_{k} x_{k-2}\ldots x_j &=x_{k-2}x_{k-4}\ldots x_{j-2}{}^\iota x_{\overline n-j}
\end{align*}
\end{lemma}
\begin{proof}
The second equality follows from repeated applications of the first.
Observe first that $x_k\in\left\langle s_j\mid j\geq k\right\rangle$, or equivalently, ${}^\iota x_{n-k}\in\left\langle s_j\mid j\leq k\right\rangle$. 
Consequently, the braid relations yields ${}^\iota x_i x_j=x_j{}^\iota x_i$ whenever $i+j\geq n+2$.
Now \begin{align*}
\hspace{4pt}   {}^\iota x_{n-k}\,x_k   &={}^\iota s_{n-k+1} {}^\iota s_{n-k} {}^\iota s_{n-k+2} {}^\iota s_{n-k+1} {}^\iota x_{n-k+2}\,x_k &\qquad\text{using \Cref{rec:x}}\\
        &=s_{k-1} s_k s_{k-2} s_{k-1} {}^\iota x_{n-k+2} x_k&\qquad\text{using \Cref{specificiota}}\\
        &=s_{k-1} s_{k-2} s_k s_{k-1} x_k {}^\iota x_{n-k+2}&\\
        &=x_{k-2} {}^\iota x_{n-k+2}                        &\qquad\text{using \Cref{rec:x}}
\end{align*}
This proves the claim when $n$ is even.
Suppose $n$ is odd, so that $\overline n=n-1$ and $k\leq n-2$.
Then \begin{align*}
\hspace{20pt}{}^\iota x_{n-k-1}\,x_k &={}^\iota s_{n-k}{}^\iota s_{n-k-1}{}^\iota x_{n-k}\,x_k  &\qquad\text{using \Cref{rec:x}}\\
                            &=s_k s_{k+1}\,x_{k-2}{}^\iota x_{n-k+2}&\qquad\text{using \Cref{specificiota}}\\
                            &=x_{k-2}s_{k-2}s_{k-1}{}^\iota x_{n-k+2}   &\qquad\text{using \Cref{rel1}}\\
                            &=x_{k-2}{}^\iota x_{n-k+1} &\qquad\text{using \Cref{rec:x}}
\end{align*}
This proves the claim when $n$ is odd.
\end{proof}

\begin{prop}
\label{intersectw}
For $w_r$ given by \Cref{form:skew}, and $v_r={}^\iota(w_0w_r\ws)$, we have 
\begin{align*}
    (w_rv_r)^{\dynk_0}=w_rv_r{}^\iota v_{\overline n-r}^{-1}={}^\iota w_{\overline n-r}\in W_{\J\cup\{\alpha_0\}}^0
\end{align*}
Consequently, ${}^\iota w_{\overline n-r}$ is the unique maximal element in $\left\{u\in W_{\J\cup\{\alpha_0\}}^0\mid u\leq (w_rv_r)^{\dynk_0}\right\}$.
\end{prop}
\begin{proof}
Let $x_i$ be as in \Cref{rec:x}.
We have the following formulae, which are easily verified inductively:\begin{align*}
    x_i         &=[1,\ldots,i-1,i+2,\ldots,n-2,2n-i,2n-i+1]\\
    w_r         &=x_{r-1} x_{r-3}\ldots x_1 \\
    w_0w_r\ws   &=x_{\overline n-1}x_{\overline n-3}\ldots x_{r+1}
\end{align*}
Now\begin{align*}\hspace{14pt}
    {}^\iota w_r{}^\iota v_r&={}^\iota x_{r-1}{}^\iota x_{r-3}\ldots{}^\iota x_1 x_{\overline n-1}x_{\overline n-3}\ldots x_{r+1}\\
                            &=x_{\overline n-r-1}x_{\overline n-r-3}\ldots x_1\,{}^\iota x_{\overline n-1}{}^\iota x_{\overline n-3}\ldots{}^\iota x_{\overline n-r+1}&\qquad\text{using \Cref{rel2}.}\\
                            &=w_{\overline n-r} v_{\overline n-r}
\end{align*}
It follows that ${}^\iota w_r{}^\iota v_rv_{\overline n-r}^{-1}=w_{\overline n-r}$, hence $w_rv_r{}^\iota v_{\overline n-r}^{-1}={}^\iota w_{\overline n-r}$.

Next, \Cref{wrbelongs} yields\begin{align*}
    w_{\overline n-r}\in W_0^{\J\cup\{\alpha_0\}}\implies {}^\iota w_{\overline n-r}\in W_{\J\cup\{\alpha_0\}}^0\subset W^{\dynk_0}
\end{align*}
Further, since ${}^\iota v_{\overline n-r}\in W_0$ (see \Cref{wlj}), we have $w_rv_r={}^\iota w_{\overline n-r}\,\mod\, W_0$.
Together, we deduce $(w_rv_r)^{\dynk_0}={}^\iota w_{\overline n-r}$.
\end{proof}